\documentclass[12pt,a4paper,leqno]{amsart}

\usepackage[latin1]{inputenc}
\usepackage[T1]{fontenc}
\usepackage{amsfonts}
\usepackage{amsmath}
\usepackage{amssymb}
\usepackage{eurosym}
\usepackage{mathrsfs}
\usepackage{palatino}

\newcommand{\R}{\mathbb{R}}
\newcommand{\C}{\mathbb{C}}

\newcommand{\Z}{\mathbb{Z}}

\newcommand{\ave}[1]{\langle #1 \rangle}
\numberwithin{equation}{section}

\swapnumbers
\theoremstyle{plain}
\newtheorem{thm}[equation]{Theorem}
\newtheorem{lem}[equation]{Lemma}
\newtheorem{prop}[equation]{Proposition}

\theoremstyle{definition}
\newtheorem{defn}[equation]{Definition}

\newtheorem{exmp}[equation]{Example}

\theoremstyle{remark}
\newtheorem{rem}[equation]{Remark}

\title{On general local $Tb$ theorems}

\author{Tuomas Hytönen \and Henri Martikainen}\thanks{The authors are supported by the Academy of Finland through the project ``$L^p$ methods in harmonic analysis''.}
\address{Department of Mathematics and Statistics, University of Helsinki, P.O.B. 68, FI-00014 Helsinki, Finland}
\email{tuomas.hytonen@helsinki.fi, henri.martikainen@helsinki.fi}

\subjclass[2000]{42B20 (Primary);  42B25 (Secondary)}
\keywords{Calder\'on--Zygmund operator, non-homogeneous theory, square function}

\begin{document}
\maketitle

\begin{abstract}
In this paper, local $Tb$ theorems are studied both in the doubling and non-doubling situation. We prove a local $Tb$ theorem for the class of upper doubling measures. With such general measures, scale invariant testing conditions are required ($L^{\infty}$ or BMO).
In the case of doubling measures, we also modify the general non-homogeneous method of proof to yield a new proof of the local $Tb$ theorem with $L^2$ type testing conditions. 
\end{abstract}

\section{Introduction}
There are multiple local $Tb$ theorems with a bit different assumptions. In these theorems, one assumes that to every cube $Q$ there exist functions $b^1_Q$ and $b^2_Q$, supported on $Q$,
so that we also know something about $Tb^1_Q$ and $T^*b^2_Q$ (where $T$ is a Calder\'on--Zygmund operator). One wishes to conclude that $T\colon L^2 \to L^2$ boundedly so that $\|T\|$ has
natural dependence on the assumptions.

The first local $Tb$ theorem is by Michael Christ \cite{Ch}, and there it was assumed that $\|b^1_Q\|_{\infty} \le C$ and $\|Tb^1_Q\|_{\infty} \le C$ (and similarly for $b^2_Q$ and $T^*b^2_Q$).
This was proven for doubling measures (even in metric spaces). Nazarov, Treil and Volberg \cite{NTVa} obtained a version of this theorem for measures satisfying the power bound $\mu(B(x,r)) \lesssim r^m$ for a given number $m$. So it is a non-homogeneous version of Christ's theorem in $\R^n$
(it also allows BMO control in the operator side if the kernel of $T$ is antisymmetric).

For doubling measures, one can also consider more general $L^p$ type testing conditions introduced by Auscher, Hofmann, Muscalu, Tao and Thiele \cite{AHMTT}, and further studied by Hofmann \cite{Ho}, Auscher and Yang \cite{AY} and Tan and Yan \cite{TY}.
The most general assumption used in these papers is of the form that $\int_Q |b^1_Q|^p \le |Q|$, $\int_Q |b^2_Q|^q \le |Q|$, $\int_Q |Tb^1_Q|^{q'} \le |Q|$ and $\int_Q |T^*b^2_Q|^{p'} \le |Q|$, where $s'$ denotes the dual exponent of $s$ and $1 < p, q \le \infty$.

In \cite{AHMTT} a theorem of this type is proved only for very special operators, the so-called perfect dyadic singular integral operators. This was expected to easily generalize
for all Calder\'on--Zygmund operators -- but this turned out not to be the case (it being easy, at least). In \cite{Ho} the theorem is extended
for standard Calder\'on--Zygmund operators, but only in the case $\int_Q |b^1_Q|^s \le |Q|$, $\int_Q |b^2_Q|^s \le |Q|$, $\int_Q |Tb^1_Q|^{2} \le |Q|$ and $\int_Q |T^*b^2_Q|^{2} \le |Q|$
for some $s > 2$. Finally, \cite{AY} establishes, by reducing the question to the known case of perfect dyadic operators,
the theorem for standard Calder\'on--Zygmund operators in the case $1/p + 1/q \le 1$. The case $1/p + 1/q > 1$ is still open for general Calder\'on--Zygmund operators.

Recently we extended global $Tb$ theorems to the general setting of upper doubling measures in metric spaces \cite{HM}. It is our opinion that the upper doubling theory constitutes a flexible framework yielding, in particular, proofs that work simultaneously for doubling and non-doubling measures. See also \cite{Hy1}, \cite{Ma1}, \cite{HYY}.

The purpose of this paper is twofold. We extend, streamline and modify the general non-homogeneous proof technique of Nazarov, Treil and Volberg \cite{NTVa} to the case of upper doubling measures $\mu$.
Also, considering doubling measures $\nu$, we show that this technique, appropriately modified, can also be used to prove local $Tb$ theorems with $L^p$ type testing conditions. As the general case $1/p + 1/q > 1$ is still quite open, a completely new
approach may be appreciated. We give such a proof in the case of $\int_Q |b^1_Q|^2\,d\nu \le \nu(Q)$, $\int_Q |b^2_Q|^2\,d\nu \le \nu(Q)$, $\int_Q |Tb^1_Q|^s\,d\nu \le \nu(Q)$ and $\int_Q |T^*b^2_Q|^s\,d\nu \le \nu(Q)$ for any $s>2$.
So this is basically the theorem of Hofmann \cite{Ho}, but with a general doubling measure and with the extra integrability assumptions pushed into the operator side, allowing to demand less integrability from the test functions.

Regarding non-homogeneous analysis in this local situation, there seems to be a previously unnoticed problem with the use of goodness and the implication it may have on the collapse of certain paraproducts. In any case, we add the good cubes into the decomposition in a new way so that no such problem may arise.

While giving finishing touches to this paper, we also learned about a recent related manuscript by Auscher and Routin \cite{AR}. Using the local $T1$ theorem, the so-called {BCR} algorithm, and Hardy inequalities, some partial progress on the case $1/p +  1/q > 1$ is achieved there. However, many technical assumptions appear. Furthermore, they also establish a direct proof (in the sense that it is not a reduction to the perfect dyadic case like \cite{AY}) of the case $1/p + 1/q \le 1$.

Our proof of the local $Tb$ theorem with $L^2$ test functions is also direct in many senses. As it uses this general non-homogeneous proof technique, it does not rely on the standard local $T1$ theorem unlike all the other known proofs seem to do. Also, it is not a reduction to the perfect dyadic case. The use of the Hardy type inequalities is completely replaced by the use of non-homogeneous analysis. Such techniques
also circumvent many problems with nearby cubes and boundary regions.

\subsection*{Acknowledgements}
Parts of this paper benefited from the interaction with a related on-going project of the first author with Antti V\"ah\"akangas. We would like to thank him for this fruitful exchange of ideas.

\section{Definitions and the main result}
\subsection{Upper doubling measures and Calder\'on--Zygmund operators}
Let $\lambda\colon \R^n \times (0,\infty) \to (0,\infty)$ be a function so that $r \mapsto \lambda(x,r)$ is non-decreasing and $\lambda(x, 2r) \le C_{\lambda}\lambda(x,r)$ for
all $x \in \R^n$ and $r > 0$. Let $\mu$ be a Borel measure in $\R^n$. We assume that $\mu$ is upper doubling with the dominating function $\lambda$, that is, $\mu(B(x,r)) \le
\lambda(x,r)$ for all $x \in \R^n$ and $r > 0$. In the case of doubling measures one can take $\lambda(x,r) = \mu(B(x,r))$, and in the case of power bounded measures ($\mu(B(x,r)) \le Cr^m$), one can take $\lambda(x,r) = Cr^m$.
Let $d = \log_2 C_{\lambda}$ -- this is a convenient number for us, and can be thought of as a dimension of the measure $\mu$.

The kernel estimates are always tied to the particular choice of $\lambda$. We say that $K \colon \R^n \times \R^n \setminus \{(x,y): x = y\} \to \C$ is a standard kernel if there holds
for some $C < \infty$ and $\alpha > 0$ that
\begin{displaymath}
|K(x,y)| \le C\min\Big(\frac{1}{\lambda(x, |x-y|)}, \frac{1}{\lambda(y, |x-y|)}\Big), \qquad x \ne y,
\end{displaymath}
\begin{displaymath}
|K(x,y) - K(x',y)| \le C\frac{|x-x'|^{\alpha}}{|x-y|^{\alpha}\lambda(x, |x-y|)}, \qquad |x-y| \ge 2|x-x'|,
\end{displaymath}
and
\begin{displaymath}
|K(x,y) - K(x,y')| \le C\frac{|y-y'|^{\alpha}}{|x-y|^{\alpha}\lambda(y, |x-y|)}, \qquad |x-y| \ge 2|y-y'|.
\end{displaymath}
These are the familiar standard estimates in the case of doubling and power bounded measures. Let $\gamma = \alpha/(2\alpha + 2d)$ -- another convenient number.

Sometimes the property $\lambda(x, |x-y|) \sim \lambda(y, |x-y|)$ would be convenient. This can be arranged as follows. In \cite[Proposition 1.1]{HYY} it is shown that
$\Lambda(x,r) := \inf_{z \in \R^n} \lambda(z, r + |x-z|)$ satisfies that $r \mapsto \Lambda(x,r)$ is non-decreasing, $\Lambda(x,2r) \le C_{\lambda}\Lambda(x,r)$,
$\mu(B(x,r)) \le \Lambda(x,r)$, $\Lambda(x,r) \le \lambda(x,r)$ and $\Lambda(x,r) \le C_{\lambda}\Lambda(y,r)$ if $|x-y| \le r$. Even the kernel estimates hold
with $\Lambda$, since $1/\lambda \le 1/\Lambda$. Thus, we may (and do) assume that $\lambda$ satisfies the additional symmetry property
$\lambda(x,r) \le C\lambda(y,r)$ if $|x-y| \le r$, and then demand the kernel estimates in the form
\begin{displaymath}
|K(x,y)| \le \frac{C}{\lambda(x, |x-y|)}, \qquad x \ne y,
\end{displaymath}
\begin{displaymath}
|K(x,y) - K(x',y)| \le C\frac{|x-x'|^{\alpha}}{|x-y|^{\alpha}\lambda(x, |x-y|)}, \qquad |x-y| \ge 2|x-x'|,
\end{displaymath}
and
\begin{displaymath}
|K(x,y) - K(x,y')| \le C\frac{|y-y'|^{\alpha}}{|x-y|^{\alpha}\lambda(x, |x-y|)}, \qquad |x-y| \ge 2|y-y'|.
\end{displaymath}

A Calder\'on--Zygmund operator with a standard kernel $K$ is a bounded linear operator $T$ taking $L^2(\mu)$ into $L^2(\mu)$ so that there holds
\begin{displaymath}
Tf(x) = \int K(x,y)f(y)\,d\mu(y)
\end{displaymath}
for $x$ not in the support of $f$.

Note that while we assume the boundedness of $T$ a priori, we are interested in quantitative bounds for $\|T\|$, which only depend on some specified information.

\subsection{Systems of accretive functions}
When working with a general upper doubling measure $\mu$, we assume that to every cube $Q \subset \R^n$ there exist two functions $b^1_Q$ and $b^2_Q$ so that there holds
\begin{itemize}
\item[(i)] spt$\,b^1_Q \subset Q$, spt$\,b^2_Q \subset Q$;
\item[(ii)] $\|b^1_Q\|_{L^{\infty}(\mu)} \le C$, $\|b^2_Q\|_{L^{\infty}(\mu)} \le C$;
\item[(iii)] $\|Tb^1_Q\|_{L^{\infty}(\mu)} \le C$, $\|T^*b^2_Q\|_{L^{\infty}(\mu)} \le C$;
\item[(iv)] $\int_Q b^1_Q \,d\mu = \mu(Q) = \int_Q b^2_Q\,d\mu$.
\end{itemize}
We call these \emph{accretive $L^{\infty}$ systems}.
At least in the case of an antisymmetric kernel, one could make do with BMO control in the operator side (see \cite{NTVa}), but we focus only on this case.

When working with a doubling measure $\nu$, we may also use the following set of assumptions: to every cube $Q \subset \R^n$ there exist two
functions $b^1_Q$ and $b^2_Q$ so that there holds
\begin{itemize}
\item[(i)] spt$\,b^1_Q \subset Q$, spt$\,b^2_Q \subset Q$;
\item[(ii)] $\int_Q |b^1_Q|^2\,d\nu \le C\nu(Q)$, $\int_Q |b^2_Q|^2\,d\nu \le C\nu(Q)$;
\item[(iii)] $\int_Q |Tb^1_Q|^s\,d\nu \le C\nu(Q)$, $\int_Q |T^*b^2_Q|^s\,d\nu \le C\nu(Q)$ for some fixed $s>2$;
\item[(iv)] $\int_Q b^1_Q = \nu(Q) = \int_Q b^2_Q$.
\end{itemize}
We call these \emph{accretive $L^2$ systems} (suppressing from the name the fact that we actually impose the somewhat stronger $L^s$ conditions in (iii)).

We now formulate our main theorem.

\begin{thm}
Let $\mu$ be an upper doubling measure with a dominating function $\lambda$ and $T\colon L^2(\mu) \to L^2(\mu)$ a Calder\'on--Zygmund operator with a standard kernel $K$.
Assuming the existence of accretive $L^{\infty}$ systems $(b^1_Q)$ and $(b^2_Q)$, we have $\|T\| \le C$, where
$C$ depends on the dimension $n$ and on the explicit constants in the definitions of $\lambda$, $K$, $(b^1_Q)$ and $(b^2_Q)$.

If $\mu = \nu$ for some doubling measure $\nu$, then the same conclusion holds assuming only the existence of accretive $L^2$ systems $(b^1_Q)$ and $(b^2_Q)$.
\end{thm}

The rest of this paper contains a direct proof which simultaneously gives the theorem with either set of assumptions. In particular, the proof is neither a reduction to a local $T1$ theorem, nor to a perfect dyadic case.
Notation-wise we work so that we use $\mu$ as long as everything works with the use of either set of assumptions, and sometimes
write $\mu = \nu$ when we explicitly estimate differently in the doubling $L^2$ case. We write $X \lesssim Y$ to mean
$X \le CY$ with some constant $C$ like in the theorem. Also, $X \sim Y$ means $Y \lesssim X \lesssim Y$. Sometimes
we absorb other parameters, but then it is either explicitly said or written in the notation (e.g. $X \lesssim_{\delta} Y$ would mean
$X \le C(\delta)Y$).
\begin{rem}
While the second part of the theorem concerning the $L^2$ test function case is not new, the proof is. Certainly some new ideas are still needed to establish the theorem with general $p$ and $q$. However,
the point is not solely in the range of exponents. For example, we point out that the non-homogeneous proof technique completely avoids the use of the so called Hardy type inequalities used and studied in \cite{AR}.
\end{rem}

\section{Preliminaries}
We begin by recording the following basic facts. Let a dyadic system $\mathcal{D}$ be given. The side length of a cube $Q\in\mathcal{D}$ is denoted $\ell(Q)$, and
$Q^{(j)}$ denotes the unique cube $S \in \mathcal{D}$ for which $Q \subset S$ and $\ell(S) = 2^j\ell(Q)$. We also set
$\langle f \rangle_Q = \mu(Q)^{-1} \int_Q f \,d\mu$.
\begin{defn}
We say that a sequence $(a_Q)_{Q \in \mathcal{D}}$ of positive numbers is a Carleson sequence, if there holds
\begin{displaymath}
\mathop{\sum_{S \in \mathcal{D}}}_{S \subset Q} a_S \le C\mu(Q)
\end{displaymath}
for every $Q \in \mathcal{D}$. The condition is called the Carleson (measure) condition.
\end{defn}
The following is the famous Carleson embedding theorem.
\begin{thm}
Given a Carleson sequence $(a_Q)$ there holds for any $f \in L^2(\mu)$ that
\begin{displaymath}
\sum_{Q \in \mathcal{D}} a_Q|\langle f \rangle_Q|^2 \le C\|f\|_{L^2(\mu)}^2.
\end{displaymath}
\end{thm}
The following is called the unweighted square function estimate.
\begin{thm}\label{usfe}
There holds for any $f \in L^2(\mu)$ that 
\begin{displaymath}
\sum_{Q \in \mathcal{D}} |\langle f \rangle_Q - \langle f \rangle_{Q^{(1)}}|^2\mu(Q) \le C\|f\|_{L^2(\mu)}^2.
\end{displaymath}
\end{thm}

\subsection{Stopping times and the martingale difference operators $\Delta_Q$}
Let $\mathcal{D}$ be a dyadic system of cubes, and let $Q_0 \in \mathcal{D}$ be a fixed large cube. Let $\mathcal{D}^0 = \{Q_0\}$.

\subsubsection{Stopping time: $L^{\infty}$ case}
Let $\mathcal{D}^1 = \{Q^k_1\}_k$ consist of the maximal $\mathcal{D}$-cubes $Q \subset Q_0$ for which there holds
\begin{displaymath}
\Big| \int_Q b^1_{Q_0}\,d\mu \Big| < \mu(Q)/2.
\end{displaymath}
One easily checks that
\begin{displaymath}
\mu\Big( \bigcup_k Q^k_1 \Big) \le \tau \mu(Q_0)
\end{displaymath}
for some $\tau < 1$.

Next, one fixes a cube $Q^k_1$ and considers all the maximal $\mathcal{D}$-cubes $Q \subset Q^k_1$ for which there holds
\begin{displaymath}
\Big| \int_Q b^1_{Q^k_1}\,d\mu \Big| < \mu(Q)/2.
\end{displaymath}
One does this for every $Q^k_1 \in \mathcal{D}^1$, and then the resulting collection of cubes is called $\mathcal{D}^2 = \{Q^k_2\}_k$. One proceeds like this to obtain collections
$\mathcal{D}^j$ for every $j$. Of course, we have the property that for every $Q \in \mathcal{D}^j$ there holds
\begin{displaymath}
\mu\Big( \bigcup_{Q' \in \mathcal{D}^{j+1}, \, Q' \subset Q} Q' \Big) \le \tau\mu(Q).
\end{displaymath}

\subsubsection{Stopping time: $L^2$ case}
Define
\begin{displaymath}
M_{\nu}f(x) = \sup_{r > 0} \frac{1}{\nu(B(x,r))}\int_{B(x,r)} |f(y)|\,d\nu(y).
\end{displaymath}
Let $\mathcal{D}^1 = \{Q^k_1\}_k$ consist of the maximal $\mathcal{D}$-cubes $Q \subset Q_0$ for which there holds
\begin{displaymath}
\int_Q |M_{\nu}b^1_{Q_0}|^2\,d\nu > \delta^{-1}\nu(Q)
\end{displaymath}
or
\begin{displaymath}
\int_Q |Tb^1_{Q_0}|^s\,d\nu > \delta^{-1}\nu(Q)
\end{displaymath}
or
\begin{displaymath}
\Big| \int_Q b^1_{Q_0} \,d\nu \Big|< \delta\nu(Q).
\end{displaymath}
Fixing $\delta$ to be small enough, one easily checks that
\begin{displaymath}
\nu\Big( \bigcup_k Q^k_1 \Big) \le \tau \nu(Q_0)
\end{displaymath}
for some $\tau < 1$. This is then continued just like in the $L^{\infty}$ case.

\subsubsection{Martingale difference operators}
For every $Q \subset Q_0$ we let $Q^a$ be the smallest cube in the family $\bigcup \mathcal{D}^j$ containing $Q$. Note that if $Q \subset Q_0$ is such that
$Q^a \in \mathcal{D}^t$, there holds for every $j \ge 1$ that
\begin{displaymath}
\mu\Big( \bigcup_{Q' \in \mathcal{D}^{t+j}, \, Q' \subset Q} Q'\Big) = \sum_{Q' \in \mathcal{D}^{t+j}, \, Q' \subset Q} \mu(Q') \le \tau^{j-1}\mu(Q).
\end{displaymath}
We state a very useful (but immediate) consequence of this as a lemma.
\begin{lem}\label{mescar}
The following is a Carleson sequence: $\alpha_Q = 0$ if $Q$ is not from $\bigcup_j \mathcal{D}^j$, and it equals $\mu(Q)$ otherwise.
\end{lem}
Given a cube $Q$ let ch$(Q)$ consist of those cubes $Q' \subset Q$ for which $\ell(Q') = \ell(Q)/2$.
Let $f$ be a function supported on $Q_0$. We define
\begin{displaymath}
\Delta_Q f = \sum_{Q' \in \, \textrm{ch}(Q)} \Big[\frac{\langle f \rangle_{Q'}}{\langle b^1_{(Q')^a}\rangle_{Q'}}b^1_{(Q')^a} - \frac{\langle f \rangle_Q}{\langle b^1_{Q^a}\rangle_Q}b^1_{Q^a}\Big]\chi_{Q'}.
\end{displaymath}
Note that then we have
\begin{displaymath}
(\Delta_Q)^*f = \sum_{Q' \in \, \textrm{ch}(Q)} \Big[\frac{\langle b^1_{(Q')^a}f \rangle_{Q'}}{\langle b^1_{(Q')^a}\rangle_{Q'}} - \frac{\langle b^1_{Q^a}f \rangle_Q}{\langle b^1_{Q^a}\rangle_Q}\Big]\chi_{Q'}.
\end{displaymath}
Also set
\begin{displaymath}
E_{Q_0} = \frac{\langle f \rangle_{Q_0}}{\langle b^1_{Q_0}\rangle_{Q_0}}b^1_{Q_0}.
\end{displaymath}

\begin{lem}
The identity
\begin{equation}\label{decomp}
f = E_{Q_0}f + \sum_{Q \in \mathcal{D}} \Delta_Q f =E_{Q_0}f+\lim_{k\to\infty}\sum_{\substack{Q\in\mathcal{D} \\ \ell(Q)>2^{-k}}}\Delta_Q f
\end{equation}
holds both pointwise almost everywhere and in $L^2(\mu)$.
\end{lem}

\begin{proof}
First, few additional notations. Set $b^{a,1}_k = \sum_{Q \in \mathcal{D}_k} \chi_Q b^1_{Q^a}$ and
\begin{displaymath}
E^{a,1}_kf = \frac{E_k f}{E_k b^{a,1}_k} b^{a,1}_k,
\end{displaymath}
where naturally $E_k h = \sum_{Q \in \mathcal{D}_k} \chi_Q \langle h \rangle_Q$. It is immediate to see that the right hand side of \eqref{decomp}, for a fixed $k$, is precisely $E^{a,1}_k f$.

It follows from the stopping time construction that almost every $x\in Q_0$ belongs to only finitely many stopping cubes $P\in\bigcup_{t=0}^{\infty}\mathcal{D}^t$. If $S$ is the smallest of them, then $Q^a=S$ for all $Q\owns x$ with $\ell(Q)=2^{-k}\leq\ell(S)$. Thus $b_k^{a,1}(x)=b_S^1(x)$ and $E_k b_k^{a,1}(x)=\ave{b_S^1}_Q= E_k b_S^1(x)\to b_S^1(x)$ as $k\to\infty$ (this happens almost everywhere as the set
$\bigcup_{t=0}^{\infty}\mathcal{D}^t$ is countable). Since also $E_k f(x)\to f(x)$ almost everywhere, we have verified the pointwise convergence $E_k^{a,1}f\to f$. In the case of accretive $L^{\infty}$ systems, the $L^2(\mu)$ convergence is immediate from dominated convergence, since $|E^{a,1}_kf|\lesssim Mf$, where $M$ is the dyadic maximal operator.

It remains to prove that $E^{a,1}_k f \to f$ in $L^2(\nu)$ in the case of accretive $L^2$ systems.
Note that $\|E^{a,1}_k f\|_2 \lesssim \|f\|_2$. Thus, it suffices to prove the convergence for a given bounded function $f$. As the convergence is in any case fine in the pointwise almost everywhere sense,
we just need to find a suitable square integrable majorant. And we have
\begin{displaymath}
|E_k^{a,1} f(x)| \lesssim_f \sup_{Q = Q^a} |b^1_Q(x)| \le \Big( \sum_{Q = Q^a} |b^1_Q(x)|^2 \Big)^{1/2},
\end{displaymath}
and this is in $L^2$ since
\begin{displaymath}
\int \sum_{Q = Q^a} |b^1_Q|^2 \,d\nu = \sum_{Q=Q^a} \int _Q |b^1_Q|^2\,d\nu \lesssim \sum_{Q=Q^a} \nu(Q) \lesssim \nu(Q_0)
\end{displaymath}
by Lemma \ref{mescar}.
\end{proof}

We are usually given two dyadic systems $\mathcal{D}$ and $\mathcal{D}'$. Then we use operators $\Delta_Q$ constructed using $(b^1_Q)$ in connection with the family $\mathcal{D}$
and operators $\Delta_R$ constructed using $(b^2_R)$ in connection with the family $\mathcal{D}'$ (in the $L^2$ case the stopping time for the latter also uses $T^*$ instead of $T$, of course).
It would perhaps be better to write $\Delta_Q^1$ and $\Delta_R^2$ to indicate the difference (as we have done above for some operators that we need not use so frequently),
but we omit this for brevity. It should nevertheless be clear from the various summing conditions like $Q \in \mathcal{D}$ and $R \in \mathcal{D}'$.

\subsection{Square function estimates}
With accretive $L^{\infty}$ systems, the estimates
\begin{displaymath}
\sum_{Q \in \mathcal{D}} \|\Delta_Q f\|_2^2 \lesssim \|f\|_2^2 \qquad \textrm{and} \qquad
\sum_{Q \in \mathcal{D}} \|(\Delta_Q)^*f\|_2^2 \lesssim \|f\|_2^2
\end{displaymath}
are quite clear (see \cite[chapter 3]{NTVa}).
So in the rest of this subsection, we work with a doubling measure $\nu$ and $L^2$ type test functions (the second estimate is actually, perhaps surprisingly, generally false in this setting).

\begin{lem}\label{gcar}
The sequence
\begin{displaymath}
\beta_Q = |\langle b^1_{Q^a} \rangle_Q - \langle b^1_{Q^a} \rangle_{Q^{(1)}}|^2\nu(Q), \qquad Q \in \mathcal{D},
\end{displaymath}
is Carleson.
\end{lem}

\begin{proof}
Let $Q \in \mathcal{D}$ be such that $Q^a \in \mathcal{D}^t$.
We simply write as follows
\begin{align*}
\sum_{S \subset Q} \beta_S &= \sum_{S \subset Q}  |\langle b^1_{S^a} \rangle_S - \langle b^1_{S^a} \rangle_{S^{(1)}}|^2\nu(S) \\
&= \mathop{\sum_{S \subset Q}}_{S^a = Q^a} |\langle b^1_{Q^a} \rangle_S - \langle b^1_{Q^a} \rangle_{S^{(1)}}|^2\nu(S) \\
 &\qquad+\sum_{j=1}^{\infty} \sum_{H \in \mathcal{D}^{t+j}, \, H \subset Q} \mathop{\sum_{S \subset H}}_{S^a = H} |\langle b^1_{H} \rangle_S - \langle b^1_{H} \rangle_{S^{(1)}}|^2\nu(S)\\
 &\lesssim\|1_{Q^{(1)}}b^1_{Q^a}\|_2^2 
    +\sum_{j=1}^{\infty} \sum_{H \in \mathcal{D}^{t+j}, \, H \subset Q} \|b_H^1\|_2^2 \\
 &\lesssim\nu(Q)+\sum_{j=1}^{\infty} \sum_{H \in \mathcal{D}^{t+j}, \, H \subset Q} \nu(H)\lesssim\nu(Q)
\end{align*}
by the unweighted square function estimate (Theorem \ref{usfe}) and Lemma \ref{mescar}.
\end{proof}

\begin{prop}\label{sqest}
There holds
\begin{displaymath}
\sum_{Q \in \mathcal{D}} \|\Delta_Q f\|_2^2 \lesssim \|f\|_2^2.
\end{displaymath}
\end{prop}

\begin{proof}
Note that
\begin{displaymath}
\sum_{Q \in \mathcal{D}} \|\Delta_Q f\|_2^2 = I + II,
\end{displaymath}
where
\begin{align*}
I &= \sum_{Q \in \mathcal{D}} \mathop{\sum_{Q' \in \, \textrm{ch}(Q)}}_{(Q')^a = Q'}
\int_{Q'} \Big| \frac{\langle f \rangle_{Q'}}{\langle b^1_{Q'}\rangle_{Q'}}b^1_{Q'} - \frac{\langle f \rangle_Q}{\langle b^1_{Q^a}\rangle_Q}b^1_{Q^a}\Big|^2\,d\nu, \\
II &= \sum_{Q \in \mathcal{D}} \mathop{\sum_{Q' \in \, \textrm{ch}(Q)}}_{(Q')^a = Q^a}
\int_{Q'} \Big| \frac{\langle f \rangle_{Q'}}{\langle b^1_{Q^a}\rangle_{Q'}} - \frac{\langle f \rangle_Q}{\langle b^1_{Q^a}\rangle_Q}\Big|^2 |b^1_{Q^a}|^2\,d\nu.
\end{align*}
Furthermore, there holds (as $\nu$ is doubling) that
\begin{displaymath}
I \lesssim \sum_{Q \in \mathcal{D}} \mathop{\sum_{Q' \in \, \textrm{ch}(Q)}}_{(Q')^a = Q'} \Big[ |\langle f \rangle_{Q'}|^2 \nu(Q') + |\langle f \rangle_{Q'} - \langle f \rangle_Q|^2 \nu(Q')\Big] \lesssim \|f\|_2^2.
\end{displaymath}
Here we used Lemma \ref{mescar} to bound the first term by $\|f\|_2^2$ (the bound for the second term follows from the unweighted square function estimate, Theorem \ref{usfe}).

Next, note that
\begin{displaymath}
II \lesssim \sum_{Q \in \mathcal{D}} \mathop{\sum_{Q' \in \, \textrm{ch}(Q)}}_{(Q')^a = Q^a} \Big[ |\langle f \rangle_{Q'}|^2 |\langle b^1_{Q^a} \rangle_{Q'} - \langle b^1_{Q^a} \rangle_Q|^2 \nu(Q') +
|\langle f \rangle_{Q'} - \langle f \rangle_Q|^2 \nu(Q')\Big].
\end{displaymath}
The latter term is yet again bounded by $\|f\|_2^2$ by the unweighted square function estimate (Theorem \ref{usfe}), and the first one is, too, bounded by $\|f\|_2^2$ by the previous lemma.
\end{proof}
The following example is a bit disconcerting. After all, we want to work with accretive $L^2$ systems of functions, and the failure of such a fundamental estimate seems like a real predicament. A weaker, but sufficient for us, substitute result is offered afterwards.

\begin{exmp}\label{counterex}
The estimate
\begin{displaymath}
\sum_{Q \in \mathcal{D}} \|(\Delta_Q)^*f\|_2^2 \lesssim \|f\|_2^2
\end{displaymath}
is not, in general, true for accretive $L^2$ systems.
\end{exmp}

\begin{proof}
Consider the one-dimensional situation with $Q_0=[0,1)$ and $N\in\Z_+$ a fixed but arbitrary parameter. We construct a sequence of examples, where the constant in the dual square function estimate grows without limit as a function of $N$. Let
\begin{equation*}
  b_{[0,2^{-j})}:=b_j:=2^{(N-j)/2}\chi_{[0,2^{-N})}+\chi_{[2^{-N},2^{-j})},\quad j=0,1,\ldots,N,
\end{equation*}
and $b_Q:=\chi_Q$ for all other dyadic intervals $Q$. They satisfy $|Q|^{-1}\int_Q|b_Q|^2 dx\leq 2$, and the accretivity of these functions is not an issue; however, the normalized $L^2$ norm of $b_{[0,2^{-j})}$ on $[0,2^{-k})$ will increase as $k$ increases. With a suitable choice of the stopping parameters, it follows that the stopping cubes are precisely all the $Q_j:=[0,2^{-j})$, $j=0,1,\ldots,N$. In particular,
\begin{equation*}
   \chi_{Q_j}(\Delta_{Q_{j-1}})^*f
   =\chi_{Q_j}\Big(\frac{ \ave{b_j f}_{Q_j} }{ \ave{b_j}_{Q_j}  } - \frac{ \ave{b_{j-1} f}_{Q_{j-1}} }{ \ave{b_{j-1}}_{Q_{j-1}}  }\Big),\quad j=1,\ldots,N.
\end{equation*}
We apply this to the function $f=2^{N/2}\chi_{Q_N}$ for which
\begin{equation*}
   \frac{ \ave{b_j f}_{Q_j} }{ \ave{b_j}_{Q_j}  }=\frac{2^{j/2}}{1+2^{(j-N)/2}-2^{j-N}},
\end{equation*}
yielding, by a simple computation,
\begin{equation*}
   \chi_{Q_j}(\Delta_{Q_{j-1}})^*f \geq c 2^{j/2} \chi_{Q_j}.
\end{equation*}
Since $\|\chi_{Q_j}\|_2^2=2^{-j}$ and $\|f\|_2=1$, it follows that
\begin{equation*}
  \sum_{Q\in\mathcal{D}}\|(\Delta_Q)^*f\|_2^2
  \geq\sum_{j=1}^N\|\chi_{Q_j}(\Delta_{Q_{j-1}})^*f\|_2^2
  \geq \sum_{j=1}^N c = cN = cN\|f\|_2^2,
\end{equation*}
and this proves the impossibility of the dual square function estimate.
\end{proof}

The following weaker estimate is, however, true and still useful.

\begin{prop}\label{subs}
For general accretive $L^2$ systems, there holds
\begin{displaymath}
\mathop{\sum_{Q \in \mathcal{D}}}_{Q^a = P} \|(\Delta_Q)^*f\|_2^2 \lesssim \|\chi_Pf\|_2^2.
\end{displaymath}
\end{prop}

\begin{proof}
We write
\begin{equation*}
  \mathop{\sum_{Q \in \mathcal{D}}}_{Q^a = P} \|(\Delta_Q)^*f\|_2^2
  =\sum_{\substack{Q\in\mathcal{D}\\ Q^a=P}} \sum_{Q'\in\textrm{ch}(Q)}\Big|
      \frac{\ave{b^1_{(Q')^a}f}_{Q'}}{\ave{b^1_{(Q')^a}}_{Q'}}-\frac{\ave{b^1_{Q^a}f}_{Q}}{\ave{b^1_{Q^a}}_{Q}}\Big|^2\nu(Q')\lesssim I+II,
\end{equation*}
where
\begin{displaymath}
I = \mathop{\sum_{Q \in \mathcal{D}}}_{Q^a = P} \mathop{\sum_{Q' \in \, \textrm{ch}(Q)}}_{(Q')^a = Q'} [|\langle b^1_{Q'}f\rangle_{Q'}|^2 + |\langle b^1_{Q^a} f\rangle_{Q'}|^2
+ |\langle b^1_{Q^a}f\rangle_{Q'} - \langle b^1_{Q^a}f\rangle_Q|^2]\nu(Q')
\end{displaymath}
and
\begin{displaymath}
II = \mathop{\sum_{Q \in \mathcal{D}}}_{Q^a = P} \mathop{\sum_{Q' \in \, \textrm{ch}(Q)}}_{(Q')^a = Q^a} [|\langle b^1_{Q^a}f\rangle_{Q'}|^2|\langle b^1_{Q^a}\rangle_{Q'} - \langle b^1_{Q^a}\rangle_Q|^2
+ |\langle b^1_{Q^a}f\rangle_{Q'} - \langle b^1_{Q^a}f\rangle_Q|^2]\nu(Q').
\end{displaymath}

Let $t$ be such that $P \in \mathcal{D}^t$. Note that
\begin{displaymath}
|\langle b^1_{Q'}f\rangle_{Q'}|^2\nu(Q') \le \frac{1}{\nu(Q')} \Big( \int_{Q'} |b^1_{Q'}|^2\,d\nu\Big) \Big(\int_{Q'} |f|^2\,d\nu \Big) \lesssim \int_{Q'} |f|^2\,d\nu
\end{displaymath}
and (as $\nu$ is doubling) that
\begin{displaymath}
|\langle b^1_Pf\rangle_{Q'}|^2\nu(Q') \le \frac{1}{\nu(Q')} \Big( \int_{Q} |b^1_P|^2\,d\nu\Big) \Big(\int_{Q'} |f|^2\,d\nu \Big) \lesssim \int_{Q'} |f|^2\,d\nu
\end{displaymath}
showing that
\begin{displaymath}
\mathop{\sum_{Q \in \mathcal{D}}}_{Q^a = P} \mathop{\sum_{Q' \in \, \textrm{ch}(Q)}}_{(Q')^a = Q'} [|\langle b^1_{Q'}f\rangle_{Q'}|^2 + |\langle b^1_Pf\rangle_{Q'}|^2]\nu(Q') \lesssim \mathop{\sum_{Q' \in \mathcal{D}^{t+1}}}_{Q' \subset P} \int_{Q'} |f|^2\,d\nu \le \|\chi_Pf\|_2^2.
\end{displaymath}
Here we used the fact that $Q' \in \mathcal{D}^{t+1}$ are disjoint.

For the rest of the terms, we write
\begin{displaymath}
\langle b^1_P f \rangle_J = \Big\langle \chi_{P \setminus \cup \mathcal{D}^{t+1}} b^1_P f \Big\rangle_J + \Big\langle \mathop{\sum_{S \in \mathcal{D}^{t+1}}}_{S \subset P} \chi_S \langle b^1_P f\rangle_S \Big\rangle_J
\end{displaymath}
for $J = Q$ or $J = Q'$, where $Q^a=P$. Recalling Lemma \ref{gcar} and the unweighted square function estimate, Theorem \ref{usfe},
we have that
\begin{align*}
\mathop{\sum_{Q \in \mathcal{D}}}_{Q^a = P} & \mathop{\sum_{Q' \in \, \textrm{ch}(Q)}}_{(Q')^a = Q^a}  |\langle b^1_Pf\rangle_{Q'}|^2|\langle b^1_{Q^a}\rangle_{Q'} - \langle b^1_{Q^a}\rangle_Q|^2\nu(Q') \\
&+ \mathop{\sum_{Q \in \mathcal{D}}}_{Q^a = P} \sum_{Q' \in \, \textrm{ch}(Q)} |\langle b^1_Pf\rangle_{Q'} - \langle b^1_Pf\rangle_Q|^2\nu(Q')
\end{align*}
is dominated by
\begin{displaymath}
\|\chi_{P \setminus \cup \mathcal{D}^{t+1}} b^1_P f\|_2^2 + \Big\| \mathop{\sum_{S \in \mathcal{D}^{t+1}}}_{S \subset P} \chi_S \langle b^1_P f\rangle_S \Big\|_2^2 \lesssim \|\chi_Pf\|_2^2,
\end{displaymath}
where the last estimate follows since on $P \setminus \bigcup \mathcal{D}^{t+1}$ we have $L^{\infty}$-control of $b^1_P$ by Lebesgue's differentiation theorem, and
\begin{align*}
\Big\| \mathop{\sum_{S \in \mathcal{D}^{t+1}}}_{S \subset P} \chi_S \langle b^1_P f\rangle_S \Big\|_2^2 &= \mathop{\sum_{S \in \mathcal{D}^{t+1}}}_{S \subset P} |\langle b^1_P f\rangle_S|^2 \nu(S) \\
&\le \mathop{\sum_{S \in \mathcal{D}^{t+1}}}_{S \subset P} \frac{1}{\nu(S)} \Big( \int_{S^{(1)}} |b^1_P|^2\,d\nu \Big) \Big(\int_S |f|^2\,d\nu \Big) \\
&\lesssim  \mathop{\sum_{S \in \mathcal{D}^{t+1}}}_{S \subset P} \int_S |f|^2\,d\nu \le \|\chi_Pf\|_2^2.
\end{align*}
\end{proof}

\begin{rem}
The stronger estimate
\begin{displaymath}
\sum_{Q \in \mathcal{D}} \|(\Delta_Q)^*f\|_2^2 \lesssim \|f\|_2^2
\end{displaymath}
is true if our test functions satisfy $\int_Q |b^1_Q|^q\,d\nu \lesssim \nu(Q)$ for some $q > 2$ (and the stopping time argument is modified to use this condition, of course).
The point is that then one can cope with summing over the multiple generations of $\mathcal{D}^j$ because of the better estimate
$|\langle b^1_{Q'}f\rangle_{Q'}|^2 + |\langle b^1_{Q^a} f\rangle_{Q'}|^2 \lesssim |\langle |f|^p \rangle_{Q'}|^{2/p}$ for $p =q'< 2$ (the Hardy--Littlewood maximal function
is then bounded on $L^{2/p}$).
\end{rem}

\section{Random dyadic cubes and the decomposition of the pairing $\langle Tf, g\rangle$}\label{sec:randomcubes}
Start by fixing once and for all two compactly supported functions $f$ and $g$ so that $\|f\|_2 = \|g\|_2 = 1$ and $\|T\|/2 \le |\langle Tf, g\rangle|$. We choose a big enough integer $N$ so that
spt$\,f$, spt$\,g \subset B(0, 2^{N-3})$. Consider two independent random squares $Q_0 = Q_0(w) = w + [-2^N,2^N)^n$ and $R_0 = R_0(w') = w' + [-2^N, 2^N)^n$, where
$w, w' \in [-2^{N-1}, 2^{N-1})^n$. The cubes $Q_0$ and $R_0$ are taken to be the starting cubes of the independent grids $\mathcal{D}$ and $\mathcal{D}'$ (only the cubes inside $Q_0$ and $R_0$ matter).
Of course, the probability measure in question is the normalized Lebesgue measure on the square $[-2^{N-1}, 2^{N-1})^n$. Furthermore, note that always spt$\,f$, spt$\,g \subset \alpha Q_0 \cap \alpha R_0$ with some absolute constant $\alpha < 1$.

A cube $Q \in \mathcal{D}$ is bad (or $\mathcal{D}'$-bad), if there exists a cube $R$ in the dyadic system $\mathcal{D}'$ such that $\ell(Q) \le 2^{-r}\ell(R)$ and $d(Q, \textrm{sk}\, R) \le 2n^{1/2}\ell(Q)^{\gamma}\ell(R)^{1-\gamma}$.
Here the skeleton of $R$ is the set sk$\,R = \bigcup \partial R_i$, where $R_i$ are the children of $R$. Also, recall that
$\gamma = \alpha/(2\alpha+2d)$, where $\alpha$ is the number from the kernel estimates and $d = \log_2 C_{\lambda}$. The number $r$ is fixed to be large enough (this is quantified later).

We shall use the badness morally in the same line as it is usually used \cite{NTVa,NTV} -- the details are somewhat different, however. There are various reasons for this, and we shall carefully elaborate on those after
performing the decomposition, since this seems to us like a genuine source of trouble.

We define $\sum_{Q \in \mathcal{D}}^k = \sum_{Q \in \mathcal{D}, \, \ell(Q) > 2^{-k}}$. Using the facts that $E^{a,1}_k f \to f$ in $L^2$ and $\|E^{a,1}_k f\|_2 \lesssim \|f\|_2$ combined with dominated convergence (in the probability space) we see that
\begin{displaymath}
\langle Tf, g\rangle = \lim_{k \to \infty} E \langle T(E^{a,1}_k f), E^{a,2}_k g\rangle,
\end{displaymath}
where $E$ is the expectation over the random grids $\mathcal{D}$ and $\mathcal{D}'$; sometimes we will explicitly write it as $E=E_{\mathcal{D}}E_{\mathcal{D}'}$.
Since
\begin{equation*}
  E_k^{a,1}f=E_{Q_0}f+\sum_{Q\in\mathcal{D}}^k\Delta_Q f,
\end{equation*}
the pairing on the right hand side can be written in the form
\begin{displaymath}
\sum_{R \in \mathcal{D}'}^k \sum_{Q \in \mathcal{D}}^k \langle T(\Delta_Qf), \Delta_Rg\rangle + \langle T(E_{Q_0}f), E^{a,2}_k g \rangle + \langle T(E^{a,1}_k f), E_{R_0}g\rangle - \langle T(E_{Q_0}f), E_{R_0}g\rangle.
\end{displaymath}

Note that spt$\,E^{a,2}_k g \subset Q_0$ for all sufficiently large $k$. Thus, one can bound $|\langle T(E_{Q_0}f), E^{a,2}_k g \rangle|$ by $\mu(Q_0)^{-1/2}\|\chi_{Q_0}Tb^1_{Q_0}\|_2 \lesssim 1$. For the same reason there holds
$| \langle T(E^{a,1}_k f), E_{R_0}g\rangle| \lesssim 1$ (for large $k$). There seems to be no such equally cheap way to further bound $|\langle T(E_{Q_0}f), E_{R_0}g\rangle| \le \mu(Q_0)^{-1/2}\mu(R_0)^{-1/2}|\langle Tb^1_{Q_0}, b^2_{R_0}\rangle|$.
However, this can be controlled using a much simplified version of the arguments we shall use in Section~\ref{sec:adjacent} concerning adjacent cubes of comparable size in the main series $\sum_{R \in \mathcal{D}'}^k \sum_{Q \in \mathcal{D}}^k \langle T(\Delta_Qf), \Delta_Rg\rangle$.
We detail on this at the end of that chapter.

Therefore, one is (remembering the above remark) reduced to estimating
\begin{displaymath}
\Big| E \sum_{R \in \mathcal{D}'}^k \sum_{Q \in \mathcal{D}}^k \langle T(\Delta_Qf), \Delta_Rg\rangle \Big|
\end{displaymath}
with a bound independent of $k$. The summation after the expectation is finite, and thus all the rearrangements one could want to make are legitimate. 
In the sequel, the index $k = k_0$ is fixed, and we no longer make any reference to it in the notation. (The symbol $k$ will then be free for other uses.)

We continue to write the summation
\begin{displaymath}
\sum_{R \in \mathcal{D}'} \mathop{\sum_{Q \in \mathcal{D}}}_{\ell(Q) \le \ell(R)}
\end{displaymath}
in the form
\begin{displaymath}
\sum_{R \in \mathcal{D}'}\Big( \mathop{\mathop{\sum_{Q \in \mathcal{D}}}_{\ell(Q) \le \ell(R)}}_{d(Q,R) > 2n^{1/2}\ell(Q)^{\gamma}\ell(R)^{1-\gamma}}
+ \mathop{\mathop{\sum_{Q \in \mathcal{D}}}_{\ell(Q) \le 2^{-r}\ell(R)}}_{d(Q,R) \le 2n^{1/2}\ell(Q)^{\gamma}\ell(R)^{1-\gamma}} +
\mathop{\mathop{\sum_{Q \in \mathcal{D}}}_{2^{-r}\ell(R) < \ell(Q) \le \ell(R)}}_{d(Q,R) \le 2n^{1/2}\ell(Q)^{\gamma}\ell(R)^{1-\gamma}}\Big).
\end{displaymath}
We denote the corresponding parts of the sum by $\Sigma_i$, $i=1,2,3$.
Goodness will be separately inserted only in the middle sum $\Sigma_2$. We shall now study these sums one by one in the following sections (using both set of assumptions). Note that the sum
$\ell(R) < \ell(Q)$ will then also be in check by the symmetry of our assumptions.

\begin{rem}\label{techrem}
We now give a few technical comments to compare our strategy with previous works based on the use of random dyadic grids.
One can safely ignore these, especially if one is not too familiar with non-homogeneous analysis.

It is natural (if one follows the beautiful strategy pioneered by Nazarov, Treil and Volberg in their deep papers \cite{NTV:Cauchy}, \cite{NTV}, \cite{NTVa} and some others)
to define
\begin{displaymath}
f_{\textrm{good}} = \sum_{Q \in \mathcal{D}_{\textrm{good}}} \Delta_Q f \qquad \textrm{and} \qquad f_{\textrm{bad}} = \sum_{Q \in \mathcal{D}_{\textrm{bad}}} \Delta_Q f,
\end{displaymath}
and then write $f = f_{\textrm{good}} + f_{\textrm{bad}}$. One does the similar thing also for $g$ but using the grid $\mathcal{D}'$ and operators $\Delta_R$.
Then one decomposes
\begin{displaymath}
\langle Tf, g\rangle = \langle Tf_{\textrm{good}}, g_{\textrm{good}}\rangle + \langle Tf_{\textrm{good}}, g_{\textrm{bad}}\rangle + \langle Tf_{\textrm{bad}}, g\rangle.
\end{displaymath}

One usually wants to reduce the considerations to the pairing $\langle Tf_{\textrm{good}}, g_{\textrm{good}}\rangle$ by arguing that the bad parts are small. However, getting a hold of this smallness would typically exploit the dual square function estimate, the failure of which we already saw in our general context of accretive $L^2$ systems 
(see Example \ref{counterex}). However, with a moderate amount of work
and a certain trick we managed to show (also in the $L^2$ case) that, after all, $E\|f_{\textrm{bad}}\|_2 \lesssim c(r)\|f\|_2$, where $c(r) \to 0$ when $r \to \infty$. So this reduction could, nevertheless, always be made.

Here comes another unfortunate snag: in our local situation even the good part, as defined above, seems not so good after all. 
Let us explain. In the global $Tb$ theorems there holds $\Delta_Q f_{\textrm{good}} = \Delta_Q f$, if $Q \in \mathcal{D}_{\textrm{good}}$, and
$\Delta_Q f_{\textrm{good}} = 0$, if $Q \in \mathcal{D}_{\textrm{bad}}$. However, there is no reason for this to be true in this local situation with the more
complicated operators $\Delta_Q$, which in general fail the pairwise orthogonality $\Delta_Q\Delta_R=0$ for $Q\neq R$. This means that in the pairing
\begin{displaymath}
\langle Tf_{\textrm{good}}, g_{\textrm{good}}\rangle = \sum_{R \in \mathcal{D}'_{\textrm{good}}} \sum_{Q \in \mathcal{D}_{\textrm{good}}} \langle T(\Delta_Q f), \Delta_R g\rangle
\end{displaymath}
one cannot remove any goodness from the summation -- which one can in the global situation, if one replaces $\Delta_Q f = \Delta_Q f_{\textrm{good}}$ (and similarly for $g$), and then notes that adding some bad cubes to the sum just amounts to adding zeroes. 

One works hard to add the restriction to good cubes only, so why would one need to remove some of it?
The answer is that in the paraproduct part of the argument there is a subtle phenomenon,
where it is essential that the bigger cube has no restrictions for a certain telescoping sum to collapse. If the bigger cubes
are restricted to be good, the sum does not collapse, and the resulting object seems to be way too complicated to handle.

This is the reason why we choose to modify this earlier strategy, and insert the goodness in a different way. However, the paraproduct still does not become quite as simple as usually, and it is basically for this reason that in the $L^2$ test function case we need the stronger integrability exponent $s>2$ on the operator side.

There are subtle tricks which depend on independence to add and remove goodness, see \cite{Hy2}, \cite{Hy3} and \cite{Ma1}. These cannot be used here either, and this is
basically because $\Delta_Q$ depends not only on the cube $Q$ and its children (like in the global $Tb$ theorems), but also, through the stopping time argument, on the whole grid $\mathcal{D}$ (and this stops one from using certain independence properties).
\end{rem}
\section{Separated cubes}
The following is the long interaction lemma. For a proof in this general upper doubling situation, see \cite[Lemma 6.1 and Lemma 6.2]{HM}.
\begin{lem}
Suppose that $Q \in \mathcal{D}$ and $R \in \mathcal{D}'$ are such that $\ell(Q) \le \ell(R)$ and $d(Q,R) \ge 2n^{1/2}\ell(Q)^{\gamma}\ell(R)^{1-\gamma}$, and that $\varphi_Q$ and $\psi_R$ are $L^2(\mu)$ functions supported on $Q$ and $R$ respectively. Assume also that $\int \varphi_Q \,d\mu= 0$.
Then there holds
\begin{align*}
|\langle T\varphi_Q, \psi_R\rangle| &\lesssim \frac{\ell(Q)^{\alpha}}{d(Q,R)^{\alpha}\sup_{z \in Q} \lambda(z, d(Q,R))}\mu(Q)^{1/2}\mu(R)^{1/2}\|\varphi_Q\|_{L^2(\mu)}\|\psi_R\|_{L^2(\mu)} \\
&\lesssim \frac{\ell(Q)^{\alpha/2}\ell(R)^{\alpha/2}}{D(Q,R)^{\alpha}\sup_{z \in Q} \lambda(z, D(Q,R))}\mu(Q)^{1/2}\mu(R)^{1/2}\|\varphi_Q\|_{L^2(\mu)}\|\psi_R\|_{L^2(\mu)},
\end{align*}
where $D(Q,R) = \ell(Q) + \ell(R) + d(Q,R)$.
\end{lem}
The fact that the corresponding matrix generates a bounded operator in $\ell^2$ is the content of the next proposition (this is \cite[Proposition 6.3]{HM}).
\begin{prop}
Let
\begin{displaymath}
T_{QR} = \frac{\ell(Q)^{\alpha/2}\ell(R)^{\alpha/2}}{D(Q,R)^{\alpha}\sup_{z \in Q} \lambda(z, D(Q,R))}\mu(Q)^{1/2}\mu(R)^{1/2},
\end{displaymath}
if $Q \in \mathcal{D}$, $R \in \mathcal{D}'$ and $\ell(Q) \le \ell(R)$, and
\begin{displaymath}
T_{QR} = 0
\end{displaymath}
otherwise. Then there holds with any $x_Q, y_R \ge 0$ that
\begin{displaymath}
\sum_{Q, \, R} T_{QR}x_Qy_R \lesssim \Big( \sum_{Q} x_Q^2 \Big)^{1/2}\Big(\sum_{R}  y_R^2\Big)^{1/2}.
\end{displaymath}
\end{prop}
The above combined with the square function estimates
$\big( \sum_Q \|\Delta_Q f\|_2^2 \big)^{1/2} \lesssim \|f\|_2 = 1$ and $\big( \sum_R \|\Delta_R g\|_2^2 \big)^{1/2} \lesssim \|g\|_2 = 1$ yield the following.
\begin{prop}
There holds $|\Sigma_1| \le C$.
\end{prop}
The long range interaction lemma will still have further use to us when dealing with the sum $\Sigma_2$ in the next section.
\section{Cubes well inside another cube and the related bad part}
We shall now deal with $\Sigma_2$. We define
\begin{displaymath}
\Sigma_{2, \,\textrm{bad}} = \sum_{R \in \mathcal{D}'} \mathop{\mathop{\mathop{\sum_{Q \in \mathcal{D}}}_{\ell(Q) \le 2^{-r}\ell(R)}}_{d(Q,R) \le 2n^{1/2}\ell(Q)^{\gamma}\ell(R)^{1-\gamma}}}_{Q
\textrm{ is bad w.r.t. a cube of the size of }R \textrm{ or larger}} \langle T(\Delta_Q f), \Delta_R g\rangle.
\end{displaymath}
The last summing condition just means that there is a cube $S \in \mathcal{D}'$ such that $\ell(S) \ge \ell(R)$
and $d(Q, \textrm{sk}\,S) \le 2n^{1/2}\ell(Q)^{\gamma}\ell(S)^{1-\gamma}$. Then $\Sigma_2 = \Sigma_{2, \,\textrm{good}} + \Sigma_{2, \,\textrm{bad}}$, where
\begin{displaymath}
\Sigma_{2, \,\textrm{good}} = \sum_{R \in \mathcal{D}'} \mathop{\mathop{\mathop{\sum_{Q \in \mathcal{D}}}_{\ell(Q) \le 2^{-r}\ell(R)}}_{d(Q,R) \le 2n^{1/2}\ell(Q)^{\gamma}\ell(R)^{1-\gamma}}}_{Q
\textrm{ is good w.r.t. all the cubes of the size of }R \textrm{ or larger}} \langle T(\Delta_Q f), \Delta_R g\rangle.
\end{displaymath}

\subsection{The disposal of the bad bart $\Sigma_{2, \,\textrm{bad}}$}
Define $\mathcal{D}_{\textrm{bad}, \, A}$ to be the collection of those cubes $Q \in \mathcal{D}$ which are bad with respect to some $\mathcal{D}'$-cube of
side length $A$ or larger. We do not always explicitly write the summing conditions $\ell(Q) \le 2^{-r}\ell(R)$ and $d(Q,R) \le 2n^{1/2}\ell(Q)^{\gamma}\ell(R)^{1-\gamma}$, but these are in force, nevertheless.
We then estimate as follows
\begin{align*}
|\Sigma_{2, \,\textrm{bad}}| &\le \sum_{R \in \mathcal{D}'} \Big| \Big\langle T\Big( \mathop{\sum_{Q \in \mathcal{D}}}_{Q \in \mathcal{D}_{\textrm{bad}, \, \ell(R)}} \Delta_Q f \Big), \Delta_R g\Big\rangle \Big| \\
&\le \|T\| \sum_{R \in \mathcal{D}'} \Big\| \mathop{\sum_{Q \in \mathcal{D}}}_{Q \in \mathcal{D}_{\textrm{bad}, \, \ell(R)}} \Delta_Q f  \Big\|_2 \|\Delta_R g\|_2 \\
&= \|T\| \sum_{R \in \mathcal{D}'} \Big\| \sum_{k \ge r} \mathop{\sum_{Q \in \mathcal{D},\, \ell(Q) = 2^{-k}\ell(R)}}_{Q \in \mathcal{D}_{\textrm{bad}, \, \ell(R)}} \Delta_Q f  \Big\|_2 \|\Delta_R g\|_2 \\
&\le \|T\| \sum_{k \ge r} \sum_{R \in \mathcal{D}'} \Big( \mathop{\sum_{Q \in \mathcal{D},\, \ell(Q) = 2^{-k}\ell(R)}}_{Q \in \mathcal{D}_{\textrm{bad}, \, 2^k\ell(Q)}} \|\Delta_Q f\|_2^2\Big)^{1/2} \|\Delta_R g\|_2 \\
&\le \|T\| \sum_{k \ge r} \Big( \sum_{R \in \mathcal{D}'} \mathop{\sum_{Q \in \mathcal{D},\, \ell(Q) = 2^{-k}\ell(R)}}_{Q \in \mathcal{D}_{\textrm{bad}, \, 2^k\ell(Q)}} \|\Delta_Q f\|_2^2 \Big)^{1/2} \Big( \sum_{R \in \mathcal{D}'} \|\Delta_R g\|_2^2\Big)^{1/2} \\
&\lesssim \|T\| \|g\|_2 \sum_{k \ge r} \Big( \mathop{\sum_{Q \in \mathcal{D}}}_{Q \in \mathcal{D}_{\textrm{bad}, \, 2^k\ell(Q)}} \|\Delta_Q f\|_2^2 \Big)^{1/2},
\end{align*}
where the last estimate used Proposition \ref{sqest} and the fact that given $Q$, there are $\lesssim 1$ cubes $R$ so that $\ell(R) = 2^k\ell(Q)$ and $d(Q,R) \le 2n^{1/2}\ell(Q)^{\gamma}\ell(R)^{1-\gamma}$. Thus, we have (here the expectation $E = E_{\mathcal{D}} E_{\mathcal{D'}} = E_{w} E_{w'}$)
\begin{align*}
E|\Sigma_{2, \,\textrm{bad}}| &\lesssim \|T\| \|g\|_2 E_{\mathcal{D}} \sum_{k \ge r} \Big(\sum_{Q \in \mathcal{D}} \mathbb{P}(Q \in \mathcal{D}_{\textrm{bad}, \, 2^k\ell(Q)}) \|\Delta_Q f\|_2^2 \Big)^{1/2} \\
&\lesssim \|T\| \|g\|_2 E_{\mathcal{D}} \sum_{k \ge r} \Big( 2^{-\gamma k} \sum_{Q \in \mathcal{D}} \|\Delta_Q f\|_2^2 \Big)^{1/2} \\
&\lesssim \|T\| \|f\|_2 \|g\|_2 \sum_{k \ge r} (2^{-\gamma/2})^k = c(r) \|T\|,
\end{align*}
where $c(r) \to 0$, when $r \to \infty$ (recall $\|f\|_2 = \|g\|_2 = 1$). We now fix a large $r$ so that $E|\Sigma_{2, \,\textrm{bad}}| \le \|T\|/16$. We are done with the bad part.

\subsection{Reduction of the good part $\Sigma_{2, \,\textrm{good}}$ to a paraproduct}
Note that if $Q \in \mathcal{D}$ is good with respect to $R$ and $d(Q,R) \le 2n^{1/2}\ell(Q)^{\gamma}\ell(R)^{1-\gamma}$, then there actually is a child $R_1$ of $R$ so that
$Q \subset R_1$ and $d(Q, \partial R_1) > 2n^{1/2}\ell(Q)^{\gamma}\ell(R)^{1-\gamma}$.

\subsubsection{Case $(R_1)^a=R^a$} We begin by assuming that $(R_1)^a=R^a$.
In this case $\Delta_R g = B_{R_1}\chi_{R_1}b^2_{R^a}  +  \chi_{R \setminus R_1}\Delta_R g$.
One may then perform the usual decomposition
\begin{equation}\label{eq:decR1a=Ra}
\begin{split}
\langle T(\Delta_Q f), \Delta_R g\rangle = \langle T(\Delta_Q f), B_{R_1} b^2_{R^a}\rangle - \langle T(\Delta_Q f)&, B_{R_1}(1-\chi_{R_1})b^2_{R^a}\rangle\\
&+ \langle T(\Delta_Q f), \chi_{R \setminus R_1}\Delta_R g\rangle.
\end{split}
\end{equation}
The last term, where $\chi_{R\setminus R_1}=\sum_{i=2}^{2^n}\chi_{R_i}$, can be readily estimated using the long range interaction lemma:
\begin{align*}
|\langle T(\Delta_Q f), \chi_{R_i}\Delta_R g\rangle| &\lesssim \Big(\frac{\ell(Q)}{\ell(R)}\Big)^{\alpha/2}\mu(Q)^{1/2} \frac{\mu(R_i)^{1/2}}{\sup_{z \in Q} \lambda(z, \ell(R_i))} \|\Delta_Q f\|_2 \|\Delta_R g\|_2 \\
&\lesssim \Big(\frac{\ell(Q)}{\ell(R)}\Big)^{\alpha/2} \Big(\frac{\mu(Q)}{\mu(R_1)}\Big)^{1/2}\|\Delta_Q f\|_2 \|\Delta_R g\|_2, \qquad i \ne 1.
\end{align*}
The corresponding matrix is a bounded operator in $\ell^2$ by \cite[Lemma 6.1]{NTVa} (this is a lemma which uses no special properties of the measure). The first term will be part of the soon to be formed paraproduct.

Let us now bound the term $\langle T(\Delta_Q f), B_{R_1}(1-\chi_{R_1})b^2_{R^a}\rangle$ in the middle.
We have with any fixed $z \in Q$ that
\begin{align*}
|\langle T(\Delta_Q f), &(1-\chi_{R_1})b^2_{R^a}\rangle| \\
&=  \Big| \int_{R^a \setminus R_1} \int_Q [K(x,y) - K(x,z)] \Delta_Q f(y)b^2_{R^a}(x)\,d\mu(y)\,d\mu(x)\Big| \\
&\lesssim \|\Delta_Q f\|_1 \int_{R^a \setminus R_1} \frac{\ell(Q)^{\alpha}}{|x-z|^{\alpha}\lambda(z, |x-z|)} |b^2_{R^a}(x)|d\mu(x).
\end{align*}
Let us first bound this in the easier case of the $L^{\infty}$ test functions.
We have that
\begin{align*}
\int_{R^a \setminus R_1} \frac{\ell(Q)^{\alpha}}{|x-z|^{\alpha}\lambda(z, |x-z|)} |b^2_{R^a}(x)|d\mu(x) &\lesssim \ell(Q)^{\alpha}\int_{\R^n \setminus B(z, d(Q, \partial R_1))} \frac{|x-z|^{-\alpha}}{\lambda(z,|x-z|)}d\mu(x) \\
&\lesssim \ell(Q)^{\alpha}d(Q, \partial R_1)^{-\alpha} \lesssim \Big(\frac{\ell(Q)}{\ell(R)}\Big)^{\alpha/2},
\end{align*}
where we used \cite[Lemma 2.4]{HM} and the fact that $d(Q, \partial R_1) \gtrsim \ell(Q)^{1/2}\ell(R)^{1/2}$.

Let us now establish the same bound in the case of $L^2$ test functions (we do not even need a doubling measure for this -- so this gives another proof of the above estimate too). Here we need to use the fact that $Q$ is good with respect to $R$ and all the bigger cubes.
Let $M$ be such that $(R_1)^{(M+1)} = R^a$. We have
\begin{align*}
\int_{R^a \setminus R_1} &\frac{\ell(Q)^{\alpha}}{|x-z|^{\alpha}\lambda(z, |x-z|)} |b^2_{R^a}(x)|d\mu(x) \\
&= \ell(Q)^{\alpha} \sum_{j=0}^M \int_{(R_1)^{(j+1)} \setminus (R_1)^{(j)}} \frac{|b^2_{R^a}(x)|}{|x-z|^{\alpha}\lambda(z, |x-z|)}\,d\mu(x).
\end{align*}
There holds (since $\gamma(\alpha +d) = \alpha/2$) that
\begin{align*}
  |x-z|^{\alpha}\lambda(z, |x-z|)
  &\gtrsim d(Q,\partial R_1^{(j)})^{\alpha}\lambda(z,d(Q,\partial R_1^{(j)})) \\
  &\gtrsim \ell(Q)^{\gamma \alpha}\ell((R_1)^{(j)})^{\alpha - \gamma\alpha} \Big(\frac{\ell((R_1)^{(j)})}{\ell(Q)}\Big)^{-\gamma d} \mu((R_1)^{(j+1)})\\
  &= \ell(Q)^{\alpha/2}\ell((R_1)^{(j)})^{\alpha/2} \mu((R_1)^{(j+1)}),
\end{align*}
and so using the fact that $\int_{(R_1)^{(j+1)}} |b^2_{R^a}| \,d\mu \lesssim \mu((R_1)^{(j+1)})$, we have
\begin{displaymath}
\int_{R^a \setminus R_1} \frac{\ell(Q)^{\alpha}}{|x-z|^{\alpha}\lambda(z, |x-z|)} |b^2_{R^a}(x)|d\mu(x) \lesssim \Big(\frac{\ell(Q)}{\ell(R)}\Big)^{\alpha/2} \sum_{j=0}^{\infty} 2^{-\alpha j/2} \lesssim \Big(\frac{\ell(Q)}{\ell(R)}\Big)^{\alpha/2}.
\end{displaymath}

As $|B_{R_1}| \lesssim \mu(R_1)^{-1/2} \|\Delta_R g\|_2$, we have shown that
\begin{displaymath}
|\langle T(\Delta_Q f), B_{R_1}(1-\chi_{R_1})b^2_{R^a}\rangle| \lesssim \Big(\frac{\mu(Q)}{\mu(R_1)}\Big)^{1/2} \Big( \frac{\ell(Q)}{\ell(R)}\Big)^{\alpha/2} \|\Delta_Q f\|_2 \|\Delta_R g\|_2,
\end{displaymath}
and this is known to be acceptable (see again \cite[Lemma 6.1]{NTVa}).

\subsubsection{Case $(R_1)^a = R_1$}
We then assume that $(R_1)^a = R_1$. In this case we write $B_{R_1} = \langle g \rangle_{R_1} / \langle b^2_{R_1} \rangle_{R_1}$ and $C_R = \langle g \rangle_R / \langle b^2_{R^a} \rangle_R$, and then decompose as follows:
\begin{equation}\label{eq:decR1a=R1}
\begin{split}
\langle T(\Delta_Q f), \Delta_R g\rangle = \langle T&(\Delta_Q f), B_{R_1}b^2_{R_1} \rangle - \langle T(\Delta_Q f), C_R b^2_{R^a} \rangle \\
&+ C_R\langle T(\Delta_Q f), (1-\chi_{R_1})b^2_{R^a} \rangle + \langle T(\Delta_Q f), \chi_{R \setminus R_1}\Delta_R g\rangle.
\end{split}
\end{equation}

The last term, being identical to the last term in \eqref{eq:decR1a=Ra}, is again handled using the long range interaction lemma. The next to last term is also estimated as above, except that this time we have $|C_R| \lesssim |\langle g \rangle_R|$, so we get
\begin{displaymath}
|C_R\langle T(\Delta_Q f), (1-\chi_{R_1})b^2_{R^a} \rangle| \lesssim \Big(\frac{\mu(Q)}{\mu(R_1)}\Big)^{1/2} \Big( \frac{\ell(Q)}{\ell(R)}\Big)^{\alpha/2} \|\Delta_Q f\|_2 \mu(R_1)^{1/2} |\langle g \rangle_R|.
\end{displaymath}
This is again fine by \cite[Lemma 6.1]{NTVa}, since
\begin{displaymath}
\Big( \mathop{\sum_{R \in \mathcal{D}'}}_{(R_1)^a = R_1} |\langle g \rangle_R|^2 \mu(R_1)\Big)^{1/2} \lesssim \|g\|_2
\end{displaymath}
by Carleson's embedding theorem. The first two terms in \eqref{eq:decR1a=R1} will be part of the paraproduct.

\subsection{The paraproduct and its boundedness}
Let $\mathcal{D}_{\textrm{good},\,k}$ be the collection of those $Q \in \mathcal{D}$ which are good with respect to all $\mathcal{D}'$-cubes of side length
$2^k\ell(Q)$ and larger. If $Q \in \bigcup_{k \ge r} \mathcal{D}_{\textrm{good},\, k}$, let $\alpha(Q)$ be the smallest index $k$ so that $Q \in \mathcal{D}_{\textrm{good},\, k}$.
So collecting the terms that we did not yet estimate in \eqref{eq:decR1a=Ra} and \eqref{eq:decR1a=R1}, we see that we need to bound
\begin{displaymath}
\mathop{\sum_{Q \in \cup_{k \ge r} \mathcal{D}_{\textrm{good},\, k}}}_{Q \subset R_0} \mathop{\mathop{\sum_{R \in \mathcal{D}'}}_{\ell(R) \ge 2^{\alpha(Q)}\ell(Q)}}_{d(Q,R) \le 2n^{1/2}\ell(Q)^{\gamma}\ell(R)^{1-\gamma}} 
\Big\langle T(\Delta_Q f),
\frac{\langle g \rangle_{R_1}}{\langle b^2_{(R_1)^a}\rangle_{R_1}}b^2_{(R_1)^a} - \frac{\langle g \rangle_R }{\langle b^2_{R^a} \rangle_R }b^2_{R^a}\Big\rangle.
\end{displaymath}
Note that there is a unique $R$ of each side length in the inner sum, the one with $R \supset Q$.
In the above summation, let $S(Q) \in \mathcal{D}'$ be $R_1$, when $\ell(R) = 2^{\alpha(Q)}\ell(Q)$.
Then bringing the $R$ summation inside the pairing, we see that the sum collapses to
\begin{displaymath}
\mathop{\sum_{Q \in \cup_{k \ge r} \mathcal{D}_{\textrm{good},\, k}}}_{Q \subset R_0} \Big\langle T(\Delta_Q f), \frac{\langle g \rangle_{S(Q)}}{\langle b^2_{S(Q)^a}\rangle_{S(Q)}}b^2_{S(Q)^a} - \frac{\langle g \rangle_{R_0}}{\langle b^2_{R_0}\rangle_{R_0}}b^2_{R_0}  \Big\rangle.
\end{displaymath}
We write this in the form
\begin{displaymath}
\sum_{R \in \mathcal{D}'} \mathop{\sum_{Q \in \cup_{k \ge r} \mathcal{D}_{\textrm{good},\, k}}}_{S(Q) = R} \Big\langle T(\Delta_Q f), \frac{\langle g \rangle_R}{\langle b^2_{R^a}\rangle_R}b^2_{R^a} \Big\rangle
- \mathop{\sum_{Q \in \cup_{k \ge r} \mathcal{D}_{\textrm{good},\, k}}}_{Q \subset R_0}  \Big\langle T(\Delta_Q f), \frac{\langle g \rangle_{R_0}}{\langle b^2_{R_0}\rangle_{R_0}}b^2_{R_0}  \Big\rangle.
\end{displaymath}
So we were able to collapse the sum because we introduced the goodness in a more restricted way than is usually done (see Remark \ref{techrem}). But the result is somewhat different from the usual paraproducts, since $S(Q)$ can be arbitrarily larger than $Q$.

At this stage we bring the absolute values inside the summations. We may then consider the following, somewhat more general, situation. Let us be given a collection $\mathcal{F} \subset \mathcal{D}$ so that
to every cube $Q \in \mathcal{F}$ there is associated a unique cube $F(Q) \in \mathcal{D}'$ for which there holds $Q \subset F(Q)$. The rest of this section is concerned with proving that
\begin{displaymath}
\sum_{R \in \mathcal{D}'} \mathop{\sum_{Q \in \mathcal{F}}}_{F(Q) = R} |\langle T(\Delta_Q f), \langle g \rangle_R b^2_{R^a}\rangle| \lesssim  1.
\end{displaymath}

We begin by recalling from \cite[p. 271]{NTVa} that
\begin{displaymath}
\Delta_Qf = (\Delta_Q)^2f + \mathop{\sum_{P \in\, \textrm{ch}(Q)}}_{P^a =P} \varphi_P,
\end{displaymath}
where
\begin{displaymath}
\varphi_P = \frac{\langle f \rangle_Q}{\langle b^1_{Q^a} \rangle_Q} \Big[ \frac{\langle b^1_{Q^a} \rangle_P}{\langle b^1_P \rangle_P} b^1_P - b^1_{Q^a}\Big]\chi_P.
\end{displaymath}
It follows that always
\begin{displaymath}
\|\varphi_P\|_2 \lesssim |\langle f \rangle_Q| ( |\langle b^1_{Q^a} \rangle_P| \|b^1_P\|_2 + \|\chi_P b^1_{Q^a}\|_2).
\end{displaymath}
This can be further bounded by $|\langle f \rangle_Q| \mu(P)^{1/2}$ (in the $L^2$ case, the doubling property is needed here).

We estimate
\begin{align*}
\sum_{R \in \mathcal{D}'} \mathop{\sum_{Q \in \mathcal{F}}}_{F(Q) = R} |\langle T(\Delta_Q f), \langle g \rangle_R b^2_{R^a}\rangle| \le I + II,
\end{align*}
where
\begin{align*}
I &= \sum_{R \in \mathcal{D}'} \mathop{\sum_{Q \in \mathcal{F}}}_{F(Q) = R} |\langle \Delta_Q f, \langle g \rangle_R (\Delta_Q)^*T^*b^2_{R^a}\rangle|, \\
II &= \sum_{R \in \mathcal{D}'} \mathop{\sum_{Q \in \mathcal{F}}}_{F(Q) = R} \mathop{\sum_{P \in\, \textrm{ch}(Q)}}_{P^a =P} |\langle \varphi_P, \langle g \rangle_R \chi_P T^*b^2_{R^a}\rangle|.
\end{align*}
There holds that
\begin{align*}
I &\le \Big( \sum_{Q \in \mathcal{D}} \|\Delta_Q f\|_2^2\Big)^{1/2} \Big( \sum_{R \in \mathcal{D}'} |\langle g \rangle_R|^2 \mathop{\sum_{Q \in \mathcal{F}}}_{F(Q) = R} \|(\Delta_Q)^*T^*b^2_{R^a}\|_2^2\Big)^{1/2}\\
&\lesssim \Big( \sum_{R \in \mathcal{D}'} |\langle g \rangle_R|^2 a_R\Big)^{1/2},
\end{align*}
where $a_R = \sum_{Q \in \mathcal{F}, \,F(Q) = R} \|(\Delta_Q)^*T^*b^2_{R^a}\|_2^2$. Also, there holds
\begin{align*}
II &\le \Big(\sum_{Q \in \mathcal{D}} |\langle f \rangle_Q|^2 \mathop{\sum_{P \in \, \textrm{ch}(Q)}}_{P^a =P} \mu(P) \Big)^{1/2} \Big( \sum_{R \in \mathcal{D}'} |\langle g \rangle_R|^2 \mathop{\sum_{Q \in \mathcal{F}}}_{F(Q) = R} \mathop{\sum_{P \in\, \textrm{ch}(Q)}}_{P^a =P}
\|\chi_P T^*b^2_{R^a}\|_2^2\Big)^{1/2} \\
&\lesssim \Big( \sum_{R \in \mathcal{D}'} |\langle g \rangle_R|^2 b_R \Big)^{1/2},
\end{align*}
where $b_R = \sum_{Q \in \mathcal{F}, \,F(Q) = R} \sum_{P \in \, \textrm{ch}(Q), \, P^a = P} \|\chi_P T^*b^2_{R^a}\|_2^2$.

We are reduced to showing that $(a_R)$ and $(b_R)$ form Carleson sequences (both in the $L^{\infty}$ test function and in the $L^2$ test function case).
\begin{lem}
The sequence
\begin{displaymath}
a_R = \mathop{\sum_{Q \in \mathcal{F}}}_{F(Q) = R} \|(\Delta_Q)^*T^*b^2_{R^a}\|_2^2
\end{displaymath}
is Carleson.
\end{lem}
\begin{proof}
Consider an arbitrary $R \in \mathcal{D}'$. 
We write
\begin{align*}
\mathop{\sum_{S \in \mathcal{D}'}}_{S \subset R} a_S &= \mathop{\sum_{S \in \mathcal{D}'}}_{S \subset R} \mathop{\sum_{Q \in \mathcal{F}}}_{F(Q) = S} \|(\Delta_Q)^*T^*b^2_{S^a}\|_2^2 \\
&= \Big(\mathop{\mathop{\sum_{S \in \mathcal{D}'}}_{S \subset R}}_{S^a = R^a} + \mathop{\sum_{H \subset R}}_{H^a = H} \mathop{\sum_{S \in \mathcal{D}'}}_{S^a = H}\Big) \mathop{\sum_{Q \in \mathcal{F}}}_{F(Q) = S} \|(\Delta_Q)^*T^*b^2_{S^a}\|_2^2.
\end{align*}
We are reduced to showing that for an arbitrary $H \in \mathcal{D}'$ there holds
\begin{displaymath}
I_H := \mathop{\mathop{\sum_{S \in \mathcal{D}'}}_{S \subset H}}_{S^a = H^a} \mathop{\sum_{Q \in \mathcal{F}}}_{F(Q) = S} \|(\Delta_Q)^*T^*b^2_{H^a}\|_2^2 \lesssim \mu(H).
\end{displaymath}
We estimate as follows
\begin{displaymath}
I_H \le \mathop{\sum_{Q \in \mathcal{D}}}_{Q \subset H} \|(\Delta_Q)^*T^*b^2_{H^a}\|_2^2 = \sum_{Q \in M(H)} \mathop{\sum_{U \in \mathcal{D}}}_{U \subset Q} \|(\Delta_U)^*T^*b^2_{H^a}\|_2^2,
\end{displaymath}
where $M(H)$ consists of maximal $Q \in \mathcal{D}$ for which $Q \subset H$. Now the claim is very easy in the $L^{\infty}$ case. Just use the dual square function estimate and the fact that $\|\chi_Q T^*b^2_{H^a}\|_2^2 \lesssim \mu(Q)$.

We are thus reduced to the case $\mu = \nu$ with $L^2$ test functions. For a given $Q \in M(H)$ we estimate using Proposition \ref{subs} that
\begin{align*}
\mathop{\sum_{U \in \mathcal{D}}}_{U \subset Q} \|(\Delta_U)^*T^*b^2_{H^a}\|_2^2 &= \Big(\mathop{\mathop{\sum_{U \in \mathcal{D}}}_{U \subset Q}}_{U^a = Q^a} + \mathop{\mathop{\sum_{K \in \mathcal{D}}}_{K \subset Q}}_{K^a = K}
\mathop{\sum_{U \in \mathcal{D}}}_{U^a = K} \Big) \|(\Delta_U)^*T^*b^2_{H^a}\|_2^2 \\
&\lesssim \|\chi_Q\chi_H T^*b^2_{H^a}\|_2^2 + \mathop{\mathop{\sum_{K \in \mathcal{D}}}_{K \subset Q}}_{K^a = K} \|\chi_K\chi_H T^*b^2_{H^a}\|_2^2.
\end{align*}
Thus, there holds
\begin{align*}
I_H &\lesssim \int \Big( \sum_{Q \in M(H)} \chi_Q \Big)\chi_H |T^*b^2_{H^a}|^2\,d\nu + \int \Big( \sum_{Q \in M(H)} \mathop{\mathop{\sum_{K \in \mathcal{D}}}_{K \subset Q}}_{K^a = K} \chi_K \Big) \chi_H |T^*b^2_{H^a}|^2\,d\nu \\
&\le \Big( \Big\| \sum_{Q \in M(H)} \chi_Q \Big\|_{p'}  +  \Big\| \sum_{Q \in M(H)} \mathop{\mathop{\sum_{K \in \mathcal{D}}}_{K \subset Q}}_{K^a = K} \chi_K\Big\|_{p'} \Big)   \Big( \int_H |T^*b^2_{H^a}|^s\,d\nu \Big)^{1/p} \\
&\lesssim \Big( \Big\| \sum_{Q \in M(H)} \chi_Q \Big\|_{p'}  +  \Big\| \sum_{Q \in M(H)} \mathop{\mathop{\sum_{K \in \mathcal{D}}}_{K \subset Q}}_{K^a = K} \chi_K\Big\|_{p'} \Big) \nu(H)^{1/p},
\end{align*}
where $p = s/2 > 1$.

Note that
\begin{align*}
\Big\| \mathop{\mathop{\sum_{K \in \mathcal{D}}}_{K \subset Q}}_{K^a = K} \chi_K \Big\|_{p'}^{p'} = \Big\|& \mathop{\sum_j \mathop{\sum_{K \in \mathcal{D}^j}}_{K \subset Q}} \chi_K \Big\|_{p'}^{p'} 
\\&\le \Big( \sum_j  \Big\| \mathop{\sum_{K \in \mathcal{D}^j}}_{K \subset Q} \chi_K \Big\|_{p'} \Big)^{p'} \lesssim \Big( \sum_j  \tau^{j/p'} \Big)^{p'} \nu(Q) \lesssim \nu(Q),
\end{align*}
and so
\begin{align*}
\Big\| \sum_{Q \in M(H)} \chi_Q \Big\|_{p'}^{p'}  +  \Big\| \sum_{Q \in M(H)} \mathop{\mathop{\sum_{K \in \mathcal{D}}}_{K \subset Q}}_{K^a = K} \chi_K\Big\|_{p'}^{p'} &= \sum_{Q \in M(H)} \Big( \nu(Q)
+ \Big\| \mathop{\mathop{\sum_{K \in \mathcal{D}}}_{K \subset Q}}_{K^a = K} \chi_K \Big\|_{p'}^{p'} \Big) \\
&\lesssim \sum_{Q \in M(H)} \nu(Q) \le \nu(H).
\end{align*}
This establishes that $I_H \lesssim \nu(H)^{1/p'}\nu(H)^{1/p} = \nu(H)$, as was the goal.
\end{proof}
\begin{lem}
The sequence
\begin{displaymath}
b_R = \mathop{\sum_{Q \in \mathcal{F}}}_{F(Q) = R} \mathop{\sum_{P \in\, \textrm{ch}(Q)}}_{P^a =P} \|\chi_P T^*b^2_{R^a}\|_2^2
\end{displaymath}
is Carleson.
\end{lem}
\begin{proof}
As in the proof of the previous lemma, this reduces to showing that for an arbitrary $H \in \mathcal{D}'$ there holds
\begin{displaymath}
I_H := \mathop{\mathop{\sum_{S \in \mathcal{D}'}}_{S \subset H}}_{S^a = H^a} \mathop{\sum_{Q \in \mathcal{F}}}_{F(Q) = S} \mathop{\sum_{P \in\, \textrm{ch}(Q)}}_{P^a =P} \|\chi_P T^*b^2_{H^a}\|_2^2 \lesssim \mu(H).
\end{displaymath}
Letting $M(H)$ consist of the maximal $Q \in \mathcal{D}$ for which $Q \subset H$, we have
\begin{displaymath}
I_H \le  \mathop{\sum_{Q \in \mathcal{D}}}_{Q \subset H} \mathop{\sum_{P \in\, \textrm{ch}(Q)}}_{P^a =P} \|\chi_P T^*b^2_{H^a}\|_2^2 = \sum_{Q \in M(H)} \mathop{\sum_{U \in \mathcal{D}}}_{U \subset Q}  \mathop{\sum_{P \in\, \textrm{ch}(U)}}_{P^a =P} \|\chi_P T^*b^2_{H^a}\|_2^2.
\end{displaymath}
The $L^{\infty}$ case is again clear from this (recalling Lemma \ref{mescar}). Otherwise, we have as in the proof of the previous lemma that
\begin{align*}
I_H &\le \int \Big( \sum_{Q \in M(H)} \mathop{\sum_{U \in \mathcal{D}}}_{U \subset Q}  \mathop{\sum_{P \in\, \textrm{ch}(U)}}_{P^a =P} \chi_P\Big) \chi_H |T^*b^2_{H^a}|^2\,d\nu \\
&\lesssim \Big( \sum_{Q \in M(H)} \Big\| \mathop{\mathop{\sum_{K \in \mathcal{D}}}_{K \subset Q}}_{K^a = K} \chi_K \Big\|_{p'}^{p'} \Big)^{1/p'} \nu(H)^{1/p} \lesssim \nu(H),
\end{align*}
where $p = s/2>1$.
\end{proof}

\begin{rem}
The proofs of the previous two lemmata are the only places of the paper where we use, in the case of accretive $L^2$ systems, the stronger integrability exponent $s>2$ on the operator side. The lemmata are true with $s=2$, if one always has $\nu(F(Q)) \lesssim \nu(Q)$. Unfortunately, if $F(Q) = S(Q)$, as in the proof of the main theorem, then this does not have to be the case. It does not seem to be easy to arrange the collapse of the paraproduct in such a way that $S(Q)$ would be, say, always precisely $r$ generations larger than $Q$ (and still know how to estimate the bad part to be small).
\end{rem}

The above two lemmata end our proof of the boundedness of the paraproduct.
Recalling that $\Sigma_2 =  \Sigma_{2, \,\textrm{good}}  + \Sigma_{2, \,\textrm{bad}}$, $E|\Sigma_{2, \,\textrm{bad}}| \le \|T\|/16$, and that $\Sigma_{2, \,\textrm{good}}$ decomposes
into the paraproduct and some other terms, all of which we have shown to be bounded, we have established the following proposition.

\begin{prop}
There holds, after fixing the parameter $r$ to be large enough, that
\begin{displaymath}
|E\Sigma_2| \le C + \|T\|/16.
\end{displaymath}
\end{prop}

\section{Adjacent cubes of comparable size}\label{sec:adjacent}
We shall sum over those $Q \in \mathcal{D}$, $R \in \mathcal{D}'$ for which $2^{-r}\ell(R) < \ell(Q) \le \ell(R)$ and $d(Q,R) \le 2n^{1/2}\ell(Q)^{\gamma}\ell(R)^{1-\gamma}$.
For a given $R$, there are only boundedly many such $Q$. Thus, this reduces to considering a finite number of subseries
\begin{displaymath}
\sum_{R \in \mathcal{D}'} \langle T(\Delta_Q f), \Delta_R g\rangle,
\end{displaymath}
where $Q = Q(R)$. Moreover, one may assume that $R \mapsto Q(R)$ is invertible.

There holds
\begin{displaymath}
\langle T(\Delta_Q f), \Delta_R g\rangle =  \sum_{i,j = 1}^{2^n} \langle T(\chi_{Q_i}\Delta_Q f), \chi_{R_j}\Delta_R g\rangle.
\end{displaymath}
If $Q \in \mathcal{D}_k$, one can write
\begin{displaymath}
\chi_{Q_i}\Delta_Q f = \chi_{Q_i}b^1_{Q_i^a}\langle s_k \rangle_{Q_i} + \chi_{Q_i}b^1_{Q_i^a}\langle h_k \rangle_{Q_i} + \chi_{Q_i}b^1_{Q^a}\langle u_k \rangle_{Q_i},
\end{displaymath}
where
\begin{displaymath}
s_k = \chi_{\{b^{a,1}_{k+1} = b^{a,1}_k\}} \Big( \frac{E_{k+1}f}{E_{k+1} b^{a,1}_{k+1}} - \frac{E_k f}{E_k b^{a,1}_k}\Big),
\end{displaymath}
\begin{displaymath}
h_k = \chi_{\{b^{a,1}_{k+1} \ne b^{a,1}_k\}} \frac{E_{k+1}f}{E_{k+1} b^{a,1}_{k+1}}, \qquad u_k = -\chi_{\{b^{a,1}_{k+1} \ne b^{a,1}_k\}} \frac{E_k f}{E_k b^{a,1}_k}.
\end{displaymath}
Here we interpret
\begin{displaymath}
\{b^{a,1}_{k+1} = b^{a,1}_k\} = \mathop{\bigcup_{Q \in \mathcal{D}_{k+1}}}_{Q^a = (Q^{(1)})^a} Q, \qquad \{b^{a,1}_{k+1} \ne b^{a,1}_k\} = \mathop{\bigcup_{Q \in \mathcal{D}_{k+1}}}_{Q^a = Q} Q.
\end{displaymath}
Hence, we can dominate our series with nine summands of the form
\begin{displaymath}
\sum_{i,j =1}^{2^n} \Big| \sum_{R \in \mathcal{D}'} \langle f_Q \rangle_{Q_i} \langle T(\chi_{Q_i}\varphi_{Q, i}), \chi_{R_j}\psi_{R, j}\rangle \langle g_R\rangle_{R_j} \Big|,
\end{displaymath}
where $\langle f_Q \rangle_{Q_i} = \langle \chi_Q f_k \rangle_{Q_i} = \langle f_k \rangle_{Q_i}$ (if $Q \in \mathcal{D}_k$), and the summands are determined by the choices
\begin{displaymath}
(f_k, \varphi_{Q,i}) \in \{(s_k, b^1_{Q_i^a}), (h_k, b^1_{Q_i^a}), (u_k, b^1_{Q^a})\}
\end{displaymath}
and analogous choices for $g$. Observe that in each case we have
\begin{equation*}
  \|\chi_{Q_i}\varphi_{Q,i}\|_2+ \|\chi_{Q_i}T\varphi_{Q,i}\|_2\lesssim\mu(Q_i)^{1/2}
\end{equation*}
by the construction of the stopping time, using doubling in the case of accretive $L^2$ systems.

We fix the parameters $i, j$ now.

\begin{lem}
There holds
\begin{displaymath}
\Big( \sum_{Q \in \mathcal{D}} \mu(Q_i) |\langle f_Q \rangle_{Q_i}|^2 \Big)^{1/2} \lesssim \|f\|_2.
\end{displaymath}
\end{lem}
\begin{proof}
If $Q_i^a = Q^a$, we have $\mu(Q_i) |\langle f_Q \rangle_{Q_i}|^2 \le \|\tilde \Delta_Q f\|_2^2$, where $\tilde \Delta_Q$ is the operator $\Delta_Q$ without the multiplying $b$ functions:
\begin{displaymath}
\tilde \Delta_Q f = \sum_{Q' \in \, \textrm{ch}(Q)} \Big[\frac{\langle f \rangle_{Q'}}{\langle b^1_{(Q')^a}\rangle_{Q'}} - \frac{\langle f \rangle_Q}{\langle b^1_{Q^a}\rangle_Q}\Big]\chi_{Q'}.
\end{displaymath}
Otherwise, we have the bound
$\mu(Q_i) |\langle f_Q \rangle_{Q_i}|^2 \lesssim \mu(Q_i) [|\langle f \rangle_{Q_i}|^2 + |\langle f \rangle_Q|^2]$. Stopping cubes (those cubes $H$ for which $H^a = H$) form a Carleson sequence, and the estimate
\begin{displaymath}
\sum_{Q \in \mathcal{D}} \|\tilde \Delta_Q f\|_2^2 \lesssim \|f\|_2^2
\end{displaymath}
is shown in the same way as Proposition~\ref{sqest}, so we are done.
\end{proof}

\subsection{Surgery}
We then begin the delicate surgery part of the argument -- this is done a bit differently than in \cite{NTVa} (e.g. the concept of badly intersected cubes is not needed). Also, the $L^2$ test function case needs several modifications.

To handle the various separated terms that we shall encounter in a unified manner, estimates in the spirit of the following lemma are useful (this is a small modification of \cite[Lemma 9.3]{HM}).
\begin{lem}
Let $S_1$ and $S_2$ be two sets so that we have $d(S_1) \sim d(S_2)$ and $d(S_1, S_2) \gtrsim \delta \min(d(S_1), d(S_2))$. Suppose we are also given functions $\varphi$ and $\psi$ supported on $S_1$ and $S_2$ respectively. Then there holds that
\begin{displaymath}
|\langle T\varphi, \psi \rangle| \lesssim \delta^{-d} \|\varphi\|_2\|\psi\|_2.
\end{displaymath}
\end{lem}

Let $\eta > 0$. Define $\delta^{\eta}_Q = (1+\eta)Q \setminus (1-\eta)Q$. If $R \in \mathcal{D}'$ and $Q = Q(R) \in \mathcal{D}$, we set
\begin{displaymath}
Q_{i, \partial} = Q_i \cap \delta^{\eta}_{R_j}, \,\,\, Q_{i, s} = Q_i \setminus Q_{i, \partial} \setminus (Q_i \cap R_j), \,\,\, \Delta_{Q_i} = (Q_i \cap R_j) \setminus Q_{i, \partial},
\end{displaymath}
and define the analogous sets also for $R$. (The subscript $s$ refers to separation from $R_j$.) Of course, e.g. $Q_{i, \partial}$ depends also on $j$, but the dependence is suppressed, as $j$ is considered fixed here, in any case.
We may then decompose
\begin{align*}
\langle T(\chi_{Q_i}\varphi_{Q, i}), \chi_{R_j}\psi_{R, j}\rangle &= \langle T(\chi_{Q_i}\varphi_{Q, i}), \chi_{R_{j,s}}\psi_{R, j}\rangle \\
&+ \langle T(\chi_{Q_i}\varphi_{Q, i}), \chi_{R_{j, \partial}}\psi_{R, j}\rangle \\
&+ \langle T(\chi_{Q_{i, \partial}}\varphi_{Q, i}), \chi_{\Delta_{R_j}}\psi_{R, j}\rangle \\
&+ \langle T(\chi_{Q_{i, s}}\varphi_{Q, i}), \chi_{\Delta_{R_j}}\psi_{R, j}\rangle \\
&+ \langle T(\chi_{\Delta_{Q_i}}\varphi_{Q, i}), \chi_{\Delta_{R_j}}\psi_{R, j}\rangle.
\end{align*}

\subsubsection{Arguments involving $\eta$-boundary regions}
We always have that $\|\chi_{Q_i}\varphi_{Q,i}\|_2 \lesssim \mu(Q_i)^{1/2}$. Thus, by separation,
\begin{equation*}
  |\langle T(\chi_{Q_i}\varphi_{Q, i}), \chi_{R_{j,s}}\psi_{R, j}\rangle|
  + |\langle T(\chi_{Q_{i, s}}\varphi_{Q, i}), \chi_{\Delta_{R_j}}\psi_{R, j}\rangle| \lesssim_{\eta} \mu(Q_i)^{1/2}\mu(R_j)^{1/2}.
\end{equation*}

The relevant series with these matrix elements are then bounded by
\begin{align*}
  C(\eta)\sum_{R \in \mathcal{D}'} &\mu(Q_i)^{1/2} |\langle f_Q \rangle_{Q_i}| \cdot \mu(R_j)^{1/2} |\langle g_R\rangle_{R_j}| \\
  &\le  C(\eta)\Big( \sum_{Q \in \mathcal{D}} \mu(Q_i) |\langle f_Q \rangle_{Q_i}|^2 \Big)^{1/2} \Big( \sum_{R \in \mathcal{D}'} \mu(R_j) |\langle g_R \rangle_{R_j}|^2 \Big)^{1/2} \\
  & \lesssim C(\eta) \|f\|_2\|g\|_2 = C(\eta).
\end{align*}

Next, we have
\begin{equation*}
  |\langle T(\chi_{Q_i}\varphi_{Q, i}), \chi_{R_{j, \partial}}\psi_{R, j}\rangle| \lesssim \|T\| \mu(Q_i)^{1/2} \|\chi_{\delta^{\eta}_{Q_i}}\chi_{R_j}\psi_{R,j}\|_2. 
\end{equation*}
Thus, there holds
\begin{align*}
E_{\mathcal{D}}\Big| \sum_{R \in \mathcal{D}'} & \langle f_Q \rangle_{Q_i} \langle T(\chi_{Q_i}\varphi_{Q, i}), \chi_{R_{j, \partial}}\psi_{R, j}\rangle \langle g_R\rangle_{R_j} \Big| \\
&\le \|T\| E_{\mathcal{D}}\Big( \sum_{Q \in \mathcal{D}} \mu(Q_i) |\langle f_Q \rangle_{Q_i}|^2 \Big)^{1/2} \Big( \sum_{R \in \mathcal{D}'} \|\chi_{\delta^{\eta}_{Q_i}}\chi_{R_j}\psi_{R,j}\|_2^2 |\langle g_R \rangle_{R_j}|^2 \Big)^{1/2} \\
&\lesssim \|T\| \|f\|_2E_{\mathcal{D}}\Big( \sum_{R \in \mathcal{D}'} \|\chi_{\delta^{\eta}_{Q_i}}\chi_{R_j}\psi_{R,j}\|_2^2 |\langle g_R \rangle_{R_j}|^2 \Big)^{1/2} \\
&\lesssim \|T\| \|f\|_2 \Big( \sum_{R \in \mathcal{D}'} \|E_{\mathcal{D}}(\chi_{\delta^{\eta}_{Q_i}})\chi_{R_j}\psi_{R,j}\|_2^2 |\langle g_R \rangle_{R_j}|^2 \Big)^{1/2} \\
&\le c(\eta) \|T\| \|f\|_2 \Big( \sum_{R \in \mathcal{D}'} \mu(R_j) |\langle g_R \rangle_{R_j}|^2 \Big)^{1/2} \\
&\le c(\eta) \|T\| \|f\|_2 \|g\|_2 = c(\eta) \|T\|,
\end{align*}
where $c(\eta) \to 0$ if $\eta \to 0$. A similar estimate holds also with the matrix element $\langle T(\chi_{Q_{i, \partial}}\varphi_{Q, i}), \chi_{\Delta_{R_j}}\psi_{R, j}\rangle$.

We are left to deal with $\langle T(\chi_{\Delta_{Q_i}}\varphi_{Q, i}), \chi_{\Delta_{R_j}}\psi_{R, j}\rangle$. Choose $j(\eta) \in \Z$ so that $\eta/64 \le 2^{j(\eta)} < \eta/32$. Let $\mathcal{D}^*$
be another independent grid (e.g. choose a large cube $U_0$ at random so that $Q_0 \cup R_0 \subset U_0$ always, and use that as the starting cube of the grid $\mathcal{D}^*$).
Let $s = 2^{j(\eta)}\ell(Q_i)$ and $G = G(R) = \mathcal{D}^*_{-\log_2 s}$. 

We enlarge the sets $\Delta_{Q_i}$ and $\Delta_{R_j}$ to obtain new sets $\Delta_{Q_i}^G$ and $\Delta_{R_j}^G$ so that $\Delta_{Q_i}^G \cap \Delta_{R_j}^G = \bigcup\{g:\, g \in G, g \subset \Delta_{Q_i}^G \cap \Delta_{R_j}^G\}$. This is done so that
$\Delta_{Q_i} = \Delta_{Q_i}^G \setminus \Delta_{Q_i}^{\partial}$, $\Delta_{R_j} = \Delta_{R_j}^G \setminus \Delta_{R_j}^{\partial}$, where $\Delta_{Q_i}^{\partial} \subset Q_{i, \partial}$ and $\Delta_{R_j}^{\partial} \subset R_{j, \partial}$. Furthermore, we may
perform this so that $5g \subset Q_i \cap R_j$ if $g \subset \Delta_{Q_i}^G \cap \Delta_{R_j}^G$.

Let us now write
\begin{align*}
\langle T(\chi_{\Delta_{Q_i}}\varphi_{Q, i}), \chi_{\Delta_{R_j}}\psi_{R, j}\rangle = \langle& T(\chi_{\Delta_{Q_i}^G}\varphi_{Q, i}), \chi_{\Delta_{R_j}^G}\psi_{R, j}\rangle \\
&- \langle T(\chi_{\Delta_{Q_i}^G}\varphi_{Q, i}), \chi_{\Delta_{R_j}^{\partial}}\psi_{R, j}\rangle - \langle T(\chi_{\Delta_{Q_i}^{\partial}}\varphi_{Q, i}), \chi_{\Delta_{R_j}}\psi_{R, j}\rangle.
\end{align*}
The series which has the sum of the last two terms as its matrix element is, after averaging, dominated by $c(\eta)\|T\|$ by the very same argument used above. We fix at this point
$\eta$ to be so small that the above four $\eta$-boundary region terms contribute no more than $c(\eta)C\|T\| < \|T\|/32$.
\subsubsection{Arguments involving $\epsilon$-boundary regions}
We are reduced to consider the pairing $\langle T(\chi_{\Delta_{Q_i}^G}\varphi_{Q, i}), \chi_{\Delta_{R_j}^G}\psi_{R, j}\rangle$. Let $\epsilon > 0$, and set $G_{\epsilon} = G_{\epsilon}(R) = \bigcup_{g \in G} \delta^{\epsilon}_g$, where
$\delta^{\epsilon}_g = (1+\epsilon)g \setminus (1-\epsilon)g$.

We define $\Delta'_{Q_i} = \Delta^G_{Q_i} \cap G_{\epsilon}$ and $\tilde \Delta_{Q_i} = \Delta_{Q_i}^G \setminus G_{\epsilon}$ (and similarly for $R$).
We then write
\begin{align*}
\langle T(\chi_{\Delta_{Q_i}^G}\varphi_{Q, i}), \chi_{\Delta_{R_j}^G}\psi_{R, j}\rangle = \langle& T(\chi_{\tilde \Delta_{Q_i}}\varphi_{Q, i}), \chi_{\tilde \Delta_{R_j}}\psi_{R, j}\rangle \\
&+ \langle T(\chi_{\Delta_{Q_i}^G}\varphi_{Q, i}), \chi_{\Delta_{R_j}'}\psi_{R, j}\rangle + \langle T(\chi_{\Delta_{Q_i}'}\varphi_{Q, i}), \chi_{\tilde \Delta_{R_j}}\psi_{R, j}\rangle.
\end{align*}
There holds $|\langle T(\chi_{\Delta_{Q_i}^G}\varphi_{Q, i}), \chi_{\Delta_{R_j}'}\psi_{R, j}\rangle| \lesssim \|T\| \mu(Q_i)^{1/2} \|\chi_{G_{\epsilon}}\chi_{R_j}\psi_{R,j}\|_2$. Again, we have
$E_{\mathcal{D}^*} \chi_{G_{\epsilon}}(x) \le c(\epsilon)$, where $c(\epsilon) \to 0$ when $\epsilon \to 0$. Therefore, the series which has the sum of the last two terms as its matrix element is, after averaging, dominated by $c(\epsilon)\|T\|$.

We are now left with $\langle T(\chi_{\tilde \Delta_{Q_i}}\varphi_{Q, i}), \chi_{\tilde \Delta_{R_j}}\psi_{R, j}\rangle$. It suffices to consider
pairings $\langle T(\chi_{g_1}\chi_{\tilde \Delta_{Q_i}}\varphi_{Q, i}), \chi_{g_2}\chi_{\tilde \Delta_{R_j}}\psi_{R, j}\rangle$ for $g_1, g_2 \in G$,
as there are only boundedly many (depending on $\eta$ -- but this is fixed) cubes in $G$ which matter.
Suppose first that $g_1 \ne g_2$. Then, because of separation,
$|\langle T(\chi_{g_1}\chi_{\tilde \Delta_{Q_i}}\varphi_{Q, i}), \chi_{g_2}\chi_{\tilde \Delta_{R_j}}\psi_{R, j}\rangle| \lesssim_{\epsilon}\mu(Q_i)^{1/2}\mu(R_j)^{1/2}$.
This implies, like above, that the relevant series with this matrix element is dominated by $C(\epsilon)\|f\|_2\|g\|_2 = C(\epsilon)$.

This time we are left with $T(\chi_g\chi_{\tilde \Delta_{Q_i}}\varphi_{Q, i}), \chi_g\chi_{\tilde \Delta_{R_j}}\psi_{R, j}\rangle$ for some $g \in G$. We may write this in the form
$\langle T(\chi_H \varphi_{Q, i}), \chi_H \psi_{R, j}\rangle$, where $H = g \setminus G_{\epsilon} \subset \Delta_{Q_i}^G \cap \Delta_{R_j}^G$ is a cube (otherwise the pairing vanishes by construction). We continue to decompose
\begin{align*}
\langle T(\chi_H \varphi_{Q, i}), \chi_H \psi_{R, j}\rangle &= \langle T(\varphi_{Q, i}), \chi_H \psi_{R, j}\rangle \\
&- \langle T(\chi_{\R^n \setminus 5H} \varphi_{Q, i}), \chi_H \psi_{R, j}\rangle \\
&- \langle T(\chi_{5H \setminus (1+\epsilon)H} \varphi_{Q, i}), \chi_H \psi_{R, j}\rangle \\
&- \langle T(\chi_{(1+\epsilon)H \setminus H} \varphi_{Q, i}), \chi_H \psi_{R, j}\rangle.
\end{align*}
We have $|\langle T(\varphi_{Q, i}), \chi_H \psi_{R, j}\rangle| \le \|\chi_{Q_i}T(\varphi_{Q, i})\|_2\|\chi_{R_j}\psi_{R,j}\|_2 \lesssim \mu(Q_i)^{1/2}\mu(R_j)^{1/2}$. We have by separation (recall that $\ell(H) \sim \ell(Q_i) \sim \ell(R_j))$ and the fact that $5H \subset Q_i \cap R_j$
that $|T(\chi_{5H \setminus (1+\epsilon)H} \varphi_{Q, i}), \chi_H \psi_{R, j}\rangle| \lesssim_{\epsilon} \mu(Q_i)^{1/2}\mu(R_j)^{1/2}$. Also, there yet again holds that
$|\langle T(\chi_{(1+\epsilon)H \setminus H} \varphi_{Q, i}), \chi_H \psi_{R, j}\rangle| \lesssim \|T\| \|\chi_{G_{\epsilon}}\chi_{Q_i} \varphi_{Q,i}\|_2 \mu(R_j)^{1/2}$, which
is fine after averaging as before.

Having disposed of the terms above, we are to handle $\langle T(\chi_{\R^n \setminus 5H} \varphi_{Q, i}), \chi_H \psi_{R, j}\rangle$. We let $\tau = T(\chi_{\R^n \setminus 5H} \varphi_{Q, i})$ and
$\beta_H = \langle b_H^2/\mu(H), \tau\rangle$, and then decompose
\begin{displaymath}
\langle \tau, \chi_H \psi_{R, j}\rangle = \langle \tau - \beta_H, \chi_H \psi_{R, j}\rangle + \beta_H \int_H \psi_{R,j} \,d\mu.
\end{displaymath}
\begin{lem}
We have
\begin{displaymath}
|\langle \tau - \beta_H, \chi_H \psi_{R, j}\rangle| \lesssim \mu(Q_i)^{1/2}\mu(R_j)^{1/2}.
\end{displaymath}
\end{lem}
\begin{proof}
There holds
\begin{displaymath}
|\langle \tau - \beta_H, \chi_H \psi_{R, j}\rangle| \le \int |\chi_H(x)\psi_{R,j}(x)||\tau(x) - \beta_H|\,d\mu(x).
\end{displaymath}
We have (using $\frac{1}{\mu(H)} \int_H b_H^2\,d\mu = 1$) that
\begin{align*}
|\tau(x) - \beta_H| &= \Big|\tau(x) - \frac{1}{\mu(H)}\int_H b_H^2(y)\tau(y)\,d\mu(y)\Big| \\
&= \Big|\frac{1}{\mu(H)}\int_H b_H^2(y)[\tau(x)-\tau(y)]\,d\mu(y)\Big|\\
&\le \frac{1}{\mu(H)}\int_H |b_H^2(y)||\tau(x)-\tau(y)|\,d\mu(y).
\end{align*}
Furthermore, we have for $x,y \in H$ that
\begin{align*}
|\tau(x)-\tau(y)| &\le \int_{\R^n \setminus 5H} |K(x,z) - K(y,z)| |\varphi_{Q, i}(z)| \,d\mu(z) \\
&\lesssim \ell(H)^{\alpha} \int_{|x-z| > c\ell(H)} \frac{|\varphi_{Q, i}(z)|}{|x-z|^{\alpha}\lambda(x, |x-z|)}\,d\mu(z) \lesssim M_{\mu}\varphi_{Q,i}(x).
\end{align*}
Thus, we have
\begin{align*}
|\tau(x) - \beta_H| &\lesssim  \Big(\frac{1}{\mu(H)}\int_H |b_H^2(y)|\,d\mu(y)\Big)M_{\mu}\varphi_{Q,i}(x) \\
&\le \Big(\frac{1}{\mu(H)}\int_H |b_H^2(y)|^2\,d\mu(y)\Big)^{1/2}M_{\mu}\varphi_{Q,i}(x) \lesssim M_{\mu}\varphi_{Q,i}(x).
\end{align*}
We conclude that
\begin{align*}
|\langle \tau - \beta_H, \chi_H \psi_{R, j}\rangle| &\lesssim \int |\chi_H(x)\psi_{R,j}(x)| |\chi_H(x)M_{\mu}\varphi_{Q,i}(x)|\,d\mu(x) \\
&\le \Big( \int_{Q_i} |M_{\mu}\varphi_{Q,i}|^2\,d\mu\Big)^{1/2} \Big( \int_{R_j} |\psi_{R,j}|^2\,d\mu\Big)^{1/2} \lesssim \mu(Q_i)^{1/2}\mu(R_j)^{1/2},
\end{align*}
where the last estimate follows by noting that in the $L^{\infty}$ case $|M_{\mu}\varphi_{Q,i}| \le \|\varphi_{Q,i}\|_{\infty} \lesssim 1$, and that in the $L^2$ case this also works out by the stopping time and the doubling property of the measure.
\end{proof}

Finally, we are to deal with $\beta_H \int_H \psi_{R,j}\,d\mu$. The absolute value of this is dominated by
\begin{displaymath}
|\beta_H| \mu(H) \Big( \frac{1}{\mu(H)} \int_{H} |\psi_{R_j}|^2\,d\mu\Big)^{1/2} \lesssim |\beta_H| \mu(H),
\end{displaymath}
where we used that $H \subset Q_i \cap R_j$ and that in the doubling case $\mu = \nu$ we have $\nu(R) \lesssim \nu(R_j) \lesssim \nu(H)$. We then write
\begin{align*}
\beta_H \mu(H) &= \langle b_H^2, T(\varphi_{Q,i})\rangle 
 - \langle b_H^2, T(\chi_{5H \setminus (1+\epsilon)H} \varphi_{Q,i})\rangle \\
&\qquad- \langle b_H^2, T(\chi_{(1+\epsilon)H \setminus H} \varphi_{Q,i})\rangle 
 - \langle b_H^2, T(\chi_H \varphi_{Q,i})\rangle.
\end{align*}

We can now deal with all of these remaining terms (after which we are finally left with nothing more to estimate). Let us do this now.
Recalling that $5H\subset Q_i$, we have
\begin{align*}
  & |\langle b_H^2, T(\varphi_{Q,i})\rangle| \le \|b_H^2\|_2\|\chi_{Q_i}T(\varphi_{Q,i})\|_2 \lesssim \mu(H)^{1/2}\mu(Q_i)^{1/2}, \\
  & |\langle b_H^2, T(\chi_{5H \setminus (1+\epsilon)H} \varphi_{Q,i})\rangle| \\
  &\qquad\lesssim_{\epsilon} \Big( \int_H |b^2_H|^2\,d\mu\Big)^{1/2} \Big( \int_{Q_i} |\varphi_{Q,i}|^2\,d\mu\Big)^{1/2}\lesssim \mu(H)^{1/2}\mu(Q_i)^{1/2}, \\
  & |\langle b_H^2, T(\chi_{(1+\epsilon)H \setminus H} \varphi_{Q,i})\rangle| \lesssim \|T\| \mu(H)^{1/2} \|\chi_{G_{\epsilon}}\chi_{Q_i}  \varphi_{Q,i}\|_2
\end{align*}
and
\begin{align*}
 |\langle b_H^2, T(\chi_H \varphi_{Q,i})\rangle| 
 &= |\langle T^*b_H^2, \chi_H\varphi_{Q,i}\rangle| \\
 &\le \|\chi_HT^*b_H^2\|_2 \|\chi_{Q_i} \varphi_{Q,i}\|_2 \lesssim \mu(H)^{1/2}\mu(Q_i)^{1/2}.
\end{align*}

Thus, using $H\subset R_j$, we have shown that
\begin{align*}
\Big|\beta_H \int_H \psi_{R,j}\,d\mu\Big| &\le C(\epsilon)\mu(H)^{1/2}\mu(Q_i)^{1/2} + C\|T\| \mu(H)^{1/2} \|\chi_{G_{\epsilon}}\chi_{Q_i}  \varphi_{Q,i}\|_2 \\
&\le C(\epsilon)\mu(R_j)^{1/2}\mu(Q_i)^{1/2} + C\|T\| \mu(R_j)^{1/2} \|\chi_{G_{\epsilon}}\chi_{Q_i}  \varphi_{Q,i}\|_2.
\end{align*}
Furthermore, there holds
\begin{equation*}
  E_{\mathcal{D}^*}\|\chi_{G_{\epsilon}}\chi_{Q_i}  \varphi_{Q,i}\|_2
  \le c(\epsilon)\|\chi_{Q_i} \varphi_{Q_i}\|_2\  
\lesssim c(\epsilon)\mu(Q_i)^{1/2}.
\end{equation*}

We have proved that the series with the coefficient $\langle T(\chi_{\R^n \setminus 5H} \varphi_{Q, i}), \chi_H \psi_{R, j}\rangle$ can be dominated by $C(\epsilon) + c(\epsilon)\|T\|$.
We now fix $\epsilon$ to be so small (this depends on the already fixed parameter $\eta$) that all of the $\epsilon$-boundary region terms
contribute no more than $c(\epsilon)C\|T\| < \|T\|/32$.

We have proven the following proposition in this section.

\begin{prop}
There holds
\begin{displaymath}
|E\Sigma_3| \le C+ \|T\|/16.
\end{displaymath}
\end{prop}

\begin{rem}
We still detail on the term $\mu(Q_0)^{-1/2}\mu(R_0)^{-1/2}|\langle Tb^1_{Q_0}, b^2_{R_0}\rangle|$, which we did not yet estimate in Section~\ref{sec:randomcubes}.
On the right hand side of the pairing write
\begin{displaymath}
\chi_{R_0} = \chi_{R_0 \cap Q_0} + \chi_{(R_0 \setminus Q_0) \cap \delta^{\theta}_{Q_0}} + \chi_{R_0 \setminus Q_0 \setminus \delta^{\theta}_{Q_0}}
\end{displaymath}
with some small parameter $\theta > 0$. Notice that to estimate the first pairing thus formed is this time trivial, since it can be bounded by
$\|\chi_{Q_0}Tb^1_{Q_0}\|\|b^2_{R_0}\|_2$, and this is dominated by $\mu(Q_0)^{1/2}\mu(R_0)^{1/2}$. The third pairing so formed can be bounded using separation, and this yields
$C(\theta)\mu(Q_0)^{1/2}\mu(R_0)^{1/2}$. Finally, we have
\begin{align*}
  E_{\mathcal{D}} [ \mu(Q_0)^{-1/2}\mu(R_0)^{-1/2}|\langle Tb^1_{Q_0}, \chi_{(R_0 \setminus Q_0) \cap \delta^{\theta}_{Q_0}} b^2_{R_0}\rangle|]
  &\le C\|T\|E_{\mathcal{D}}[\mu(R_0)^{-1/2} \|\chi_{\delta^{\theta}_{Q_0}} b^2_{R_0}\|_2] \\
  &\le c(\theta)C\|T\| < \|T\|/12
\end{align*}
fixing $\theta$ to be small enough.
\end{rem}

\section{Completion of the proof}
Collecting the above estimates for $\Sigma_i$, $i=1,2,3$, (and for their symmetric counterparts and for the term $\mu(Q_0)^{-1/2}\mu(R_0)^{-1/2}|\langle Tb^1_{Q_0}, b^2_{R_0}\rangle|$), we have established that
\begin{align*}
\|T\|/2 &\le |\langle Tf, g\rangle| \le C + \|T\|/3,
\end{align*}
and from this we may conclude that $\|T\| \le 6C$.
\bibliographystyle{amsalpha}
\bibliography{metric}
\end{document}